\documentclass[11pt, reqno]{amsart}

\makeatletter

\def\subsection{\@startsection{subsection}{2}%
	{\parindent}{.5\linespacing}{-.5em}%
	{\normalfont\itshape\/}}
\def\subsubsection{\@startsection{subsubsection}{3}%
	{\parindent}{.5\linespacing}{-.5em}%
	{\normalfont\itshape\/}}
\def\appendix{\par\c@section\z@ \c@subsection\z@
	\gdef\theHsection{\Hy@AlphNoErr{section}}%
\let\sectionname\appendixname{}
\def\thesection{{\upshape\@Alph\c@section}}}

\newtheoremstyle{plain}{0.5\linespacing}{0.5\linespacing}{\itshape}%
	{\parindent}{\scshape}{.}{0.5em}%
	{\thmname{#1}\thmnumber{ #2}\thmnote{\normalfont{} (#3)}}
\newtheoremstyle{definition}{0.5\linespacing}{0.5\linespacing}%
	{\upshape}{\parindent}%
	{\itshape}{.}{0.5em} 
	{\thmname{#1}\thmnumber{ #2}\thmnote{\normalfont{} (#3)}}
\newtheoremstyle{remark}{0.5\linespacing}{0.5\linespacing}%
	{\upshape}{\parindent}%
	{\itshape}{.}{0.5em} 
	{\thmname{#1}\thmnumber{ #2}\thmnote{\normalfont{} (#3)}}

\renewenvironment{proof}[1][\proofname]{\par 
	\pushQED{\qed}%
	\normalfont\topsep6\p@\@plus6\p@\relax
	\trivlist%
	\item[\hskip\labelsep\hskip\parindent%
	\itshape%
	#1\@addpunct{.}]\ignorespaces%
	}{ 
	\popQED\endtrivlist\@endpefalse%
}
\makeatother
\allowdisplaybreaks

\usepackage{amsmath, amssymb, amsthm}
\usepackage{enumitem}
\makeatletter
\renewenvironment{abstract}{%
  \global\setbox\abstractbox=\vtop \bgroup
    \normalfont\Small
    \list{}{\labelwidth\z@
      \leftmargin1pc \rightmargin\leftmargin 
      \listparindent\normalparindent \itemindent\z@
      \parsep\z@ \@plus\p@
      
    }%
    \item[\hskip\labelsep\scshape\abstractname.]%
}{%
  \endlist\egroup
  \ifx\@setabstract\relax \@setabstracta \fi
}
\makeatother

\usepackage{xcolor}
\definecolor{cite}{rgb}{0.30,0.60,1.00}
\definecolor{url}{rgb}{0.00,0.00,0.80}
\definecolor{link}{rgb}{0.40,0.10,0.20}

\usepackage[pdfusetitle,colorlinks,linkcolor=link,urlcolor=url,citecolor=cite,breaklinks,bookmarksdepth=3,bookmarksopen=true]{hyperref}
\usepackage[top=1.75in, headsep=0.25in, left=1.65in, right=1.7in,bottom=1.275in]{geometry}
\usepackage{url}
\usepackage{graphicx}
\usepackage{caption}
\usepackage{subcaption}
\usepackage{mathdots}
\usepackage{tikz-cd}
\usepackage{array}

\usepackage{mathtools}
\usepackage{physics}
\usepackage{float}
\usepackage[mathcal]{eucal}
\usepackage{microtype}
\usepackage{pagesel}
\usepackage{quiver}
\usepackage{aliascnt}
\usepackage[capitalise]{cleveref}

\theoremstyle{plain}

\newaliascnt{theorem}{counter}
\newtheorem{theorem}[theorem]{Theorem}
\aliascntresetthe{theorem}

\newaliascnt{claim}{counter}

\aliascntresetthe{claim}

\newaliascnt{proposition}{counter}
\newtheorem{proposition}[proposition]{Proposition}
\aliascntresetthe{proposition}

\newaliascnt{lemma}{counter}
\newtheorem{lemma}[lemma]{Lemma}
\aliascntresetthe{lemma}

\newaliascnt{conjecture}{counter}

\aliascntresetthe{conjecture}

\newaliascnt{corollary}{counter}
\newtheorem{corollary}[corollary]{Corollary}
\aliascntresetthe{corollary}

\theoremstyle{definition}

\newaliascnt{definition}{counter}
\newtheorem{definition}[definition]{Definition}
\aliascntresetthe{definition}

\newaliascnt{notation}{counter}

\aliascntresetthe{notation}

\newaliascnt{example}{counter}
\newtheorem{example}[example]{Example}
\aliascntresetthe{example}

\theoremstyle{remark}

\newaliascnt{remark}{counter}
\newtheorem{remark}[remark]{Remark}
\aliascntresetthe{remark}

\newcommand{\nNatural}{\mathbb{N}}
\newcommand{\zIntegers}{\mathbb{Z}}
\newcommand{\qRationals}{\mathbb{Q}}
\newcommand{\rReal}{\mathbb{R}}
\newcommand{\rRealPos}{\rReal_{+}}
\newcommand{\cComplex}{\mathbb{C}}

\renewcommand{\abs}[1]{\left|#1\right|}

\newcommand{\comp}{\mathbin{\circ}}

\let\setminusaux\setminus{}	
\renewcommand*{\setminus}{\scalemath{\mathrel}{0.7}{\setminusaux}}

\newcommand{\st}{\mid}
\newcommand{\restrict}[2]{{#1}|_{#2}}

\DeclareMathOperator{\graph}{graph}
\newcommand{\sign}{\operatorname{sign}}
\renewcommand{\exp}{\operatorname{exp}}

\renewcommand{\leq}{\leqslant}
\renewcommand{\geq}{\geqslant}

\let\tmp\epsilon{}
\let\epsilon\varepsilon{}
\let\varepsilon\tmp{}

\let\tmp\phi{}
\let\phi\varphi{}
\let\varphi\tmp{}

\let\emptyset\varnothing

\newcommand{\structure}{\mathcal{S}}
\newcommand{\RrPfaff}{\rReal_{\mathrm{rPfaff}}}
\newcommand{\Rexp}{\rReal_{\exp}}
\newcommand{\Rre}{\rReal^{\operatorname{RE}}}
\newcommand{\Rexpsin}{\Rre_{\exp}}
\newcommand{\Ran}{\rReal_{\mathrm{an}}}
\newcommand{\Ranexp}{\rReal_{\mathrm{an},\exp}}
\newcommand{\Ralg}{\rReal_{\mathrm{alg}}}

\newcommand{\s}{\raisebox{.5ex}{\scalebox{0.6}{\#}}}
\newcommand{\so}{\s\kern-.02em{}o}
\NewDocumentCommand{\format}{sO{F}}{
	\IfBooleanTF{#1}
		{#2, \ell}
		{#2}
	}
\newcommand{\degree}[1][D]{#1}
\NewDocumentCommand{\FD}{sooo}{
	\IfBooleanTF{#1}								
		{											
			\IfNoValueTF{#4}
			{\mathcal{S}_{\order[#2]{1},\poly_{#2}\!\qty(#3)}}
			{
				{\mathcal{S}^{#4}_{\order[#2]{1},\poly_{#2}\!\qty(#3)}}
			}
		}
		{											
		\IfNoValueTF{#2}
			{\mathcal{S}}
			{
				\IfNoValueTF{#4}
				{
					{\mathcal{S}_{#2,#3}}
				}
				{\mathcal{S}^{#4}_{#2,#3}}
				}		
		}
}
\NewDocumentCommand{\polyfd}{oo}{
	\IfValueTF{#1}
		{\poly_{\format[#1]}\qty(\degree[#2])}
		{\poly_{\format}\qty(\degree)}
}

\newcommand{\cell}[1][C]{\mathcal{#1}}
\newcommand{\hyperbolicParameter}[1]{\{#1\}}

\newcommand{\disc}[1][\relax]{
	\ifx\relax#1 
		D
	\else
		D\qty(#1)
	\fi
}

\NewDocumentCommand{\puncDisc}{so}{
	\IfBooleanTF{#1}
		{D_{\infty}}
		{D_{\circ}}
	\IfNoValueF{#2}{
		\qty(#2)
		}
}

\NewDocumentCommand{\annulus}{oo}{
	\IfNoValueTF{#1} 	
		{A} 			
		{A\qty(#1,#2)} 	
}
\RenewDocumentCommand{\circle}{o}{
	\IfNoValueTF{#1}{S}{S\qty(#1)}	
}

\NewDocumentCommand{\ext}{smm}{
	\IfBooleanTF{#1}		
		{\qty(#2){}^{#3}} 	
		{#2^{#3}}			
}
\NewDocumentCommand{\hExt}{smm}{
	\IfBooleanTF{#1}			
		{\qty(#2){}^{\hyperbolicParameter{#3}}} 
		{#2^{\hyperbolicParameter{#3}}}			
}
\newcommand{\initial}[3]{#1_{#2..#3}}
\NewDocumentCommand{\nuCover}{mo}{
	\IfNoValueTF{#2}
		{{#1}_{\times\nu}}
		{{#1}_{\times#2}}
}

\newcommand{\poly}{\operatorname{poly}}
\newcommand{\polyl}{\poly_{\ell}}
\newcommand{\varX}{\mathbf{x}}
\newcommand{\varZ}{\mathbf{z}}

\NewDocumentCommand{\polydisc}{soo}{
	\IfBooleanTF{#1}					
		{\Delta_{#2}\times\Delta_{#3}}	
		{\IfNoValueTF{#2}				
			{\Delta}
			{\Delta_{#2}}
		}
}
\RenewDocumentCommand{\order}{som}{
	\IfBooleanTF{#1}
		{
		\IfNoValueTF{#2}
			{\Omega\qty(#3)}
			{\Omega_{#2}\qty(#3)}
		}	
		{
		\IfNoValueTF{#2}
			{O\qty(#3)}
			{O_{#2}\qty(#3)}
		}	
}
\NewDocumentCommand{\ball}{ooo}{
	B
	\IfValueT{#1}
		{
			\qty(\IfValueTF{#2}
					{#1, #2}
					{#1}
				\IfValueT{#3}
					{; #3})
		}
}
\NewDocumentCommand{\diam}{mo}
	{
		\operatorname{diam}\qty(#1
		\IfValueT{#2}
			{; #2})
	}
\NewDocumentCommand{\dist}{mmo}
	{
		\operatorname{dist}\qty(#1,#2
		\IfValueT{#3}
			{\,; #3})
	}
\NewDocumentCommand{\latticeComplement}{o}{
	\IfNoValueTF{#1}
		{\cComplex\setminus\zIntegers^2}
		{\cComplex\setminus{\! #1}\zIntegers^2}
}

\newcommand{\LE}{\textrm{LE}}
\newcommand{\LA}{\textrm{L$\structure$}}

\makeatletter
\newdimen\scalemath@axis{}
\newcommand*{\scalemath}[3]{%
  #1{%
    \mathpalette{\scalemath@aux{#2}}{#3}%
  }%
}
\newcommand*{\scalemath@aux}[3]{%
  \begingroup
    \everyvbox{}%
    \settoheight\scalemath@axis{$#2\vcenter{}$}%
    \raisebox{\scalemath@axis}{%
      \scalebox{#1}{%
        \raisebox{-\scalemath@axis}{%
          $\m@th#2#3$%
        }%
      }%
    }%
  \endgroup
}
\makeatother 

\title[Log--exp preparation in sharply o-minimal structures]{Logarithmic--exponential preparation in sharply o-minimal structures}

\author{Gal Binyamini}
\address{Weizmann Institute of Science, Rehovot, Israel}
\email{gal.binyamini@weizmann.ac.il}

\author{Oded Carmon}
\address{Weizmann Institute of Science, Rehovot, Israel}
\email{oded.carmon@weizmann.ac.il}

\author{Dmitry Novikov}
\address{Weizmann Institute of Science, Rehovot, Israel}
\email{dmitry.novikov@weizmann.ac.il}

\date{\today}

\begin{document}

\begin{abstract}
  We develop a complex-analytic approach to the model theory of the (real) unrestricted exponential.
  As a consequence we derive sharp forms of many of the foundational results for the structure ${\mathbb R}^\text{RE}_\text{exp}$ (and more general structures).
  In particular we establish sharp o-minimality, a sharp form of Wilkie's theorem of the complement, a sharp form of Wilkie's conjecture and a sharp form of piecewise definability by terms.
  Our approach is based on a complexification of the LE-preparation theorem of Lion--Rolin.

  We also develop a parallel complex theory for ${\mathbb R}_\text{an,exp}$, proving for example that the rational points of height $H$ on a nowhere-dense definable set can be interpolated by an algebraic hypersurface of degree $\text{poly}(\log H)$.
  This generalizes a theorem of Cluckers--Pila--Wilkie who proved the same statement for ${\mathbb R}_\text{an}^\text{pow}$.   
\end{abstract}

\maketitle

\addtocontents{toc}{\protect\setcounter{tocdepth}{1}}

{\small \tableofcontents}

\section{Introduction}
\subsection{Background and main results}
\subsubsection{Sharply o-minimal structures}

Sharply o-minimal structures are a framework for arithmetically tame geometry introduced by Binyamini--Novikov--Zak in~\cite{BinyaminiNovikovZak2022}.
This framework refines the notion of an o-minimal structure by replacing its characteristic qualitative finiteness properties by effective quantitative bounds in terms of the logical complexity of definable sets.
These bounds are moreover polynomial with respect to the \emph{degrees} of the relevant sets, which are analogous to the degree of an algebraic variety, and the corresponding polynomials depend only on the \emph{format} of the sets, analogous to the ambient dimension.

It is conjectured that sets and functions appearing naturally in algebraic and arithmetic geometry are definable in sharply o-minimal structures.
Until now, the only known examples of sharply o-minimal structures were the structure $\RrPfaff$, generated by the restriction to bounded boxes of \emph{Pfaffian functions} (see~\cite[Section 1.5.3]{BinyaminiNovikovZak2022}), as well as reducts of this structure.

A corollary of our first main theorem (\Cref{thm: S exp so min}) is the following.
Let $\Rre=\rReal({\restrict{\sin}{[0,1]},\restrict{\exp}{[0,1]}})$ be the expansion of the real field generated by the restrictions of the sine and exponential functions, and let $\Rexpsin$ be its expansion by the graph of the unrestricted real exponential function.

{\renewcommand{\thetheorem}{A}
\begin{theorem}\label{thm:A}
	The structure $\Rexpsin$ is sharply o-minimal.
	Moreover, one has an effective piecewise description of definable functions by terms in a suitable language.
\end{theorem}}
More precisely, we construct in \Cref{sec: so min} a \emph{sharply o-minimal FD-filtration} for $\Rexpsin$.
This is an assignment of a format $\format$ and degree $\degree$ to each definable set in $\Rexpsin$, compatible in a suitable sense with the structure operations, such that definable subsets of $\rReal$ of format $\format$ and degree $\degree$ have at most $\polyfd$ connected components.
Every definable function $f$ in $\Rexpsin$ is piecewise given by \emph{LE-functions} (see \Cref{def: LE function}), with the number of pieces, their complexities, and the complexity of the corresponding LE-functions effectively determined in terms of the format and degree of $f$.

The description of definable functions by terms yields, in particular, an effective theorem of the complement (model completeness) in the spirit of Wilkie's theorem~\cite{Wilkie1996} as explained in \Cref{sec: effective model completeness}. Theorem~\ref{thm:A} thus establishes a sharp form of much of the classical theory of $\Ranexp$.

We remark that the theorem of the complement, unlike o-minimality, is a result specific to the unrestricted exponential: it does not seem to be true for o-minimal structures generated by arbitrary chains of Pfaffian functions. All proofs of this result therefore involve features specific to the exponential map, and to our knowledge no effective version was known even setting aside ``sharp'' bounds. 

We refer the reader to~\cite[Section 2.1]{BinyaminiCarmonNovikovComplexCells} for a review of the relevant definitions and notation we use regarding sharply o-minimal structures (including the notion of sharp cell decomposition).

\subsubsection{Interpolation of rational points and Wilkie's conjecture}

One of the most successful applications of o-minimality to number theory is the Pila--Wilkie counting theorem~\cite{PilaWilkie2006}, giving subpolynomial asymptotic upper bounds for the number of rational points of bounded (multiplicative) height outside of the \emph{algebraic part} of any set which is definable in an o-minimal structure.
While these asymptotics cannot be much improved in the generality of o-minimal structures, it was also conjectured by Wilkie in~\cite{PilaWilkie2006} that for specific structures one may obtain more precise control.
Specifically, Wilkie's conjecture states that for sets definable in $\Rexp$, one may replace the subpolynomial bounds of the Pila--Wilkie theorem by polylogarithmic bounds.
Problems of this type are one of the main motivations for sharp o-minimality, and indeed using this framework this conjecture was recently resolved by Binyamini--Novikov--Zak in~\cite{BinyaminiNovikovZak2024}.

The proof of Wilkie's conjecture in~\cite{BinyaminiNovikovZak2024} consists of two stages.
First, the desired polylogarithmic bounds are established in the structure $\RrPfaff$.
This step is formulated more generally for sharply o-minimal structures satisfying an additional condition, called \emph{sharp derivatives}, which regulates the complexity of iterated derivatives of definable functions.
This condition holds for $\RrPfaff$, but is not known to hold in general.
Then, the result is extended from $\RrPfaff$ to $\Rexp$ by an application of Wilkie's theorem of the complement.
This step is not effective, and yields polynomial bounds which are not uniform over all sets of the same complexity.

We prove Wilkie's conjecture for $\Rexp$ with these improved uniform bounds.
For a set $X\subset\rReal^n$, we write $X(g,H)$ for the set of algebraic points in $X$ of degree at most $g$ (that is, points $x\in X$ such that $[\qRationals(x_1,\dots,x_n):\qRationals]\leq g$) and multiplicative Weil height at most $H$.
We write $X^{\operatorname{alg}}$ for the union of all connected, positive dimensional, semialgebraic subsets of $X$, and set $X^{\operatorname{trans}}=X\setminus X^{\operatorname{alg}}$.
A corollary of our second main theorem (\Cref{thm: wilkie intro}) is the following.

{\renewcommand{\thetheorem}{B}
\begin{theorem}\label{thm:B}
	Let $X\subset \rReal^n$ in $\Rexpsin$ be of \emph{format} $\format$ and \emph{degree} $\degree$ (in the filtration of \Cref{thm:A}).
	Then 
	\begin{equation}
		\# X^{\operatorname{trans}}(g,H)\leq\poly_{F}(D,g,\log H).
	\end{equation}
\end{theorem}}

A main step in the proof of such point counting theorems is an interpolation result for rational points of bounded height using a small number of algebraic hypersurfaces of low degree.
We prove the following for sets in~$\Ranexp$.

{\renewcommand{\thetheorem}{C}
\begin{theorem}\label{thm:C}
	Let $X\subset T\cross [0,1]^n$ be a definable family in $\Ranexp$ of sets $X_t\subset[0,1]^n$, $t\in T$.
	Assume that $\max_t \dim X_t = m < n$.
	Then, for any $t\in T$, we have that $X_t(g,H)$ is contained in the union of at most $\poly_{X}(g,\log H)$ algebraic hypersurfaces of degree at most $\order[X]{1}\cdot g^{m+1} (\log H)^m$.
\end{theorem}}
This generalizes a similar result by Cluckers--Pila--Wilkie \cite[Theorem 2.3.1]{CluckersPilaWilkie2020} for sets definable in $\Ran^{\operatorname{pow}}$.
We note that in \cite{CluckersPilaWilkie2020}, the degree of the relevant hypersurfaces is bounded by~$(\log H)^{m/(n-m)}$.
The proofs of \Cref{thm:B,thm:C} are given in \Cref{sec: point counting}.

\subsubsection{Preparation theorems and complex cells}
\label{sec: prep thms and complex cells intro}

O-minimality of $\Ranexp$, the structure generated by globally subanalytic functions and the unrestricted exponential, as well as a more explicit description of its definable sets, was established using model theoretic methods by van den Dries and Miller and by van den Dries, Macintyre, and Marker in~\cite{vdDriesMiller1994,vdDriesMacintyreMarker1994}.
Geometric proofs were later given by Lion and Rolin in \cite{LionRolin1997}.
Lion--Rolin's constructions were adapted by van den Dries and Speissegger~\cite{vdDriesSpeissegger2002} to the more general setting of a structure $\structure(\exp)$, generated by the unrestricted exponential and a polynomially bounded o-minimal structure $\structure$ containing the restricted exponential.

This geometric approach consists of establishing \emph{preparation theorems}, expressing definable functions in these structures (piecewise) as a monomial times a unit in suitable coordinates.
Such preparation theorems for subanalytic functions were also given by Parusinski~\cite{Parusinski1994}.

The main technical tools we develop are effective versions of the Lion--Rolin logarithmic--exponential preparation theorems, depending polynomially on the degrees of the relevant functions.
Following Lion--Rolin, we will refer to these as $\LE$-preparation and $\LA$-preparation theorems (see \Cref{thm: exp split,thm: LS preparation}).
Roughly, these allow us to piecewise express definable functions (in, say, $\Rexp$) as the product of a unit, an exponential, and rational powers of a sequence of (non-vanishing) nested logarithms and translations.
For the sake of \Cref{thm:A}, the new part of this is the control on the number of pieces used and on the complexity of these expansions.
However, our methods yield additional geometric information, as we explain below, which is crucial for the proofs of \Cref{thm:B,thm:C}.

We follow the treatment in~\cite{vdDriesSpeissegger2002}, which relies on a preparation theorem for functions definable in a polynomially bounded o-minimal structure, similar to the subanalytic preparation needed in~\cite{LionRolin1997}.
We replace this by an effective version (see \Cref{thm: S-preparation}) in the setting of an \emph{analytically generated} sharply o-minimal structure $\{\FD[\format][\degree]\}$ (see~\cite{CarmonAnalyticallyGenerated}).
These are structures (equipped with FD-filtrations) generated by their collection of definable \emph{complex cells} --- holomorphic analogs of the classical cells of o-minimality.
Complex cells were introduced by Binyamini--Novikov in~\cite{BinyaminiNovikov2019} for $\Ralg$ and $\Ran$ and extended to sharply o-minimal structures by the authors in~\cite{BinyaminiCarmonNovikovComplexCells}.

The analytically generated setting is expected to include certain (conjectural) sharply o-minimal structures which are relevant for number theory but might not admit sharp derivatives in the sense of~\cite{BinyaminiNovikovZak2024}.
For example, the structure generated by restrictions of the Gamma function to compact discs.

We give here the following intuitive picture for complex cells.
A classical o-minimal cell is, roughly, a ``product'' $I_1\odot\cdots\odot I_n$ of intervals $\{I_j\}$, where the fibers are translated and scaled by continuous functions on their respective bases.
A complex cell $\cell[F]_1\odot\cdots\odot\cell[F]_n$ is composed in a similar manner from basic fibers $\{\cell[F]_i\}$ such as discs and annuli in the complex plane, where the radii of the fibers are determined by holomorphic functions of the relevant bases.

The interplay of Lion--Rolin's real methods and the theory of complex cells is essential for our diophantine applications, as we now explain (see \Cref{fig: factor through scale}).
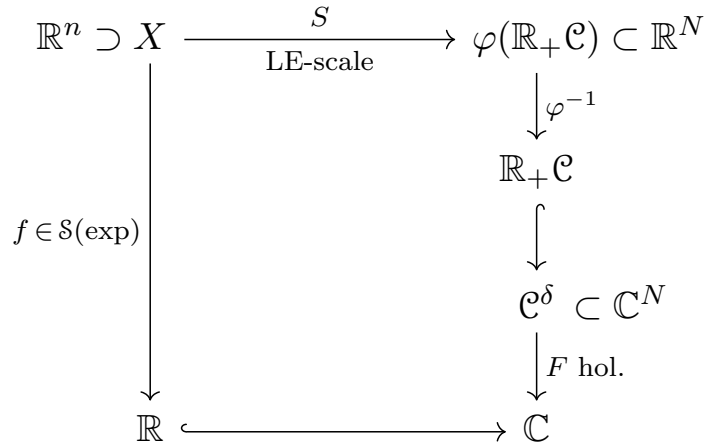
\begin{figure}[H]
\[\scalebox{1.3}{
\begin{tikzcd}[ampersand replacement=\&]
	{\mathbb{R}^n\supset\hspace{-3em}} \& X \&\&\& {\phi(\mathbb{R}_+\mathcal{C})} \& {\hspace{-3em}\subset\mathbb{R}^N} \\
	\&\&\&\& {\mathbb{R}_+\mathcal{C}} \\
	\&\&\&\& {\mathcal{C}^{\delta}} \& {\hspace{-5em}\subset\mathbb{C}^N} \\
	\& {\mathbb{R}} \&\&\& {\mathbb{C}}
	\arrow["{\text{\raisebox{\widthof{\,}}{$S$}}}", from=1-2, to=1-5]
	\arrow["{\text{\raisebox{0mm-\widthof{T}}{LE-scale}}}"', draw=none, from=1-2, to=1-5]
	\arrow["{f\,\in\,\mathcal{S}(\exp)}"', from=1-2, to=4-2]
	\arrow["{\phi^{-1}}", from=1-5, to=2-5]
	\arrow[hook, from=2-5, to=3-5]
	\arrow["{F\text{ hol.}}", from=3-5, to=4-5]
	\arrow[hook, from=4-2, to=4-5]
\end{tikzcd}
}\]
\caption{Factoring a function $f$ in $\structure(\exp)$ through a logarithmic--exponential scale $S$, a power map $\phi$, and a holomorphic function $F$ on a complex cell $\ext{\cell}{\delta}$.}
\label{fig: factor through scale}
\end{figure}
Let $f:\rReal^n\to\rReal$ be a function definable in $\structure(\exp)$.
We partition its domain into simple pieces (e.g.\ o-minimal cells in $\structure(\exp)$), such that on each piece $X$ the function $f$ factors through a \emph{(recursive) logarithmic--exponential scale} $S$ (see \Cref{def: log scale,def: log exp scale,def: recursive scale}).
This scale maps $X$ into what is essentially, up to extracting roots, the positive real part $\rRealPos\cell$ of a complex cell $\cell$.
This complex cell supports a holomorphic function $F$, which in fact extends to a larger cell $\ext{\cell}{\delta}$, such that $f$ is given by the composition $F\comp \phi^{-1}\comp S$,
where the components of $\phi$ are given by $\phi_j(\initial{\varZ}{1}{N})=\pm z_j ^{q_j}$ for some $0\ne q_j \in \zIntegers$.
The function $F$ may be taken to be either identically vanishing or else nowhere vanishing on $\ext{\cell}{\delta}$.

This process separates the ``globally subanalytic part'' of $f$, which is amenable to treatment by complex methods, from the ``logarithmic--exponential part'', which is encapsulated by the scale $S$ and which does not interact well with the complex theory.
The terms of the scale $S$ act as additional auxiliary variables for the holomorphic function $F$.
This large increase in the number of variables would normally leave us unable to proceed with the usual arguments for constructing interpolating polynomials, which depend delicately on the dimensions of the relevant sets.
However, one may construct the scale $S$ and the cell $\cell$ such that at most $\dim X$ of the coordinates of $\cell$ are given by discs, while the rest are given by, for example, annuli.
It is a crucial feature of the hyperbolic geometry of a complex cell $\cell$ inside its extension $\ext{\cell}{\delta}$ that one may then choose the extension parameters $\delta$ for the fibers of $\cell$ to be such that all non-disc coordinates are exponentially ``thinner'' than the disc coordinates, and so will not affect the required estimates.
In this way, one may treat the dimension of $\cell$ as if it were equal to $\dim X$ for the sake of constructing interpolation polynomials, which yields \Cref{thm:C}.
In the sharp setting, polynomially effective control throughout the above process yields in the same way the analogous interpolation result for $\Rre_{\exp}$ definable sets, which implies \Cref{thm:B}.

\subsection{Full statements}
We now state our main results in their general form.
Let $\structure$ be an analytically generated sharply o-minimal structure with sharp cell decomposition (see \cite{CarmonAnalyticallyGenerated}), which contains the graph of $\restrict{\exp}{[0,1]}$.
For example, one may take $\structure=\Rre=\rReal({\restrict{\sin}{[0,1]},\restrict{\exp}{[0,1]}})$.
Our first main theorem is the following.

\begin{theorem}
	\label{thm: S exp so min}
	The structure $\structure(\exp)$ generated by $\structure$ and the graph of $\exp$ admits a sharply o-minimal FD-filtration with sharp cell decomposition.
	Moreover, one has an effective piecewise description of definable functions in $\structure(\exp)$ by terms in a suitable language.
\end{theorem}
In particular, this implies \Cref{thm:A}.
The proof of this theorem is given in \Cref{sec: so min}.
We use this FD-filtration in the statements of the subsequent theorems of this section.

Our second main result is Wilkie's conjecture for $\structure(\exp)$, with $\structure$ as above and using the FD-filtration of \Cref{thm: S exp so min}.
\begin{theorem}
	\label{thm: wilkie intro}
	Let $X\subset \rReal^n$ in $\structure(\exp)$ be of format $\format$ and degree $\degree$.
	Then 
	\begin{equation}
		\# X^{\operatorname{trans}}(g,H)\leq\poly_{F}(D,g,\log H).
	\end{equation}
\end{theorem}
In particular, this implies \Cref{thm:B}.
The proof of this theorem is given in \Cref{sec: point counting}.

These two results follow from our main preparation theorems, \Cref{thm: single cell,cor: different scales poly k 2}.
As the statement is technical, we give here a simplified informal statement, omitting the precise forms of the effective bounds.
\begin{theorem}
	\label{thm: complex cell preparation intro}
	Let $f:\rReal^n\to\rReal$ be definable in $\structure(\exp)$.
	Then there is a decomposition of $\rReal^{n}$ into finitely many o-minimal cells of $\structure(exp)$, and for each cell $X$ there is a recursive log--exp scale $S$ of length $N$, a real complex cell $\ext{\cell}{{\delta}}\subset\cComplex^{N}$, a power map $\phi:\ext{\cell}{{\delta}}\to\cComplex^{N}$, and a holomorphic function $g:\ext{\cell}{{\delta}}\to\cComplex$, such that $S(X)\subset\phi(\rRealPos\cell)$ and such that ${f=g\comp\phi^{-1}\comp S}$ on $X$.
	The function $g$ may be taken to either vanish identically or else vanish nowhere on $\ext{\cell}{\delta}$.
	The complex cell $\cell$ may be taken to have at most $\dim X$-many disc fibers, corresponding to the most logarithmically nested components of $S$ with respect to each variable.
\end{theorem}
The proof of \Cref{thm: complex cell preparation intro} is given in \Cref{sec: complex cells}.
\Cref{thm:C} follows from an analogue of \Cref{thm: complex cell preparation intro} for $\Ranexp$-definable sets, given in \Cref{sec: Ranexp preparation}.

\subsection{Structure of the paper}
In order to make the proofs more accessible for the reader interested in only a subset of the results, we have tried to organize the sections in a roughly increasing order of prerequisites.

In \Cref{sec: cylindrical decomposition}, we recall some of the notions used in~\cite{vdDriesSpeissegger2002} and then construct a sharply o-minimal FD-filtration on $\structure(\exp)$ using effective cylindrical decomposition theorems for the class of \emph{$\LE$-functions} in this structure.
We then prove sharp cellular decomposition in terms of \emph{$\LE$-cells}, which gives an effective description by terms for definable functions in $\structure(\exp)$ and, additionally, an effective model completeness result in the case $\structure=\Rre$.
We also introduce the notions of complexity for $\LE$-functions which we use throughout the paper.

In \Cref{sec: prepartion theorems}, we state and prove effective preparation theorems for $\LE$-functions which, in particular, imply the cylindrical decomposition theorems used in \Cref{sec: cylindrical decomposition}.
All statements and proofs until the end of this section are formulated such that no familiarity with complex cells is needed.

In \Cref{sec: complex cells}, we prove our main preparation theorem, \Cref{thm: complex cell preparation intro}, factoring definable functions in $\structure(\exp)$ through a recursive log--exp scale and holomorphic functions on a complex cell.

Finally, in \Cref{sec: point counting}, we prove Wilkie's conjecture for~$\structure(\exp)$ (\Cref{thm: wilkie intro}) and the interpolation theorem for $\Ranexp$-definable sets (\Cref{thm:C}).
Here we crucially use the full strength of the complex cellular formulation of the preparation theorems in \Cref{sec: complex cells}.

\subsection{Notation}
Each appearance of an expression of the form $c=\poly_{a}(b)$ means that $c$ is bounded from above by the value at $x=b$ of some polynomial $p_a(x)\in\rReal_{> 0}[x]$ depending only on $a$.
This polynomial may be different at each occurrence of this notation, which only serves as shorthand.
Similarly, an expression of the form $c=\order[a]{b}$ (respectively, $c=\order*[a]{b}$) means $c\leq C_a\cdot b$ (respectively, $c\geq C_a\cdot b$), where $C_a$ is some positive constant depending only on $a$, which may be different at each occurrence of this notation.

We consider $\log$ as a function defined on all of $\rReal$ by setting its value to be~$0$ for non-positive inputs.
The size of a finite set $X$ will be denoted by $\# X$.
Finally, we will often abbreviate a tuple $(x_1,\dots,x_n)$ as $\initial{x}{1}{n}$.

\subsection*{Acknowledgements}

This work was partially done while the authors were at the Institute for Advanced Study in Princeton and they would like to thank the institute for its hospitality and for providing excellent working conditions. 
G.B.\ was supported by the Marvin V.\ and Beverly J.\ Mielke Endowed Fund and the Infosys Member Fund and D.N.\ was supported by the Kovner Member Fund.
G.B.\ and O.C.\ were also supported by the European Union (ERC, SharpOS, 101087910) and by the Israel Science Foundation (grant No.\ 2067/23).
D.N.\ was also supported by the Israel Science Foundation grant 1167/17 and by Minerva grant 714141.

An earlier version of this paper contained several minor mathematical typos that were identified by the use of an LLM.

\section{Cylindrical decomposition}
\label{sec: cylindrical decomposition}
In this section, we state a cylindrical decomposition theorem (\Cref{thm: LE cylinder decomposition}) which follows from the preparation theorems of \Cref{sec: prepartion theorems} and use it to prove our first main result, \Cref{thm: S exp so min}.
We fix from now on a sharply o-minimal structure $\{\FD[\format][\degree]\}$ which we assume to have sharp cell decomposition, to contain the graph of $\restrict{\exp}{[0,1]}$, and to be analytically generated in the sense of~\cite{CarmonAnalyticallyGenerated} (in particular, $\structure$ is a reduct of $\Ran$).

\subsection{Previous work}
We begin by recalling some of the notions used in~\cite{vdDriesSpeissegger2002}.

\begin{definition}
	\label{def: LE function}
	An \emph{$\LE$-function} is a function $f:\rReal^n\to\rReal$, for some $n\in\nNatural$, obtained by iterated composition of functions in $\structure$, of $\exp$, and of $\log$.
	We write $\LE_n$ for the collection of $\LE$-functions whose domain is $\rReal^n$.
\end{definition}
\begin{definition}
	An \emph{$\LE$-set} $S\subset\rReal^n$ is a finite union of sets of the form
	\begin{equation}
		\{\sign f_1 = \sigma_1,\dots,\sign f_k=\sigma_k\},
	\end{equation}
	where $f_1,\dots,f_k\in\LE_n$ and $\sigma_1,\dots,\sigma_k\in\{+,0,-\}$.
\end{definition}

The structure $\structure(\exp)$ is generated by $\LE$-sets.
It is clear that this collection of sets is closed under boolean operations and under taking complements.
In~\cite{vdDriesSpeissegger2002} it is shown that the projection of an $\LE$-set is again an $\LE$-set, and thus that $\LE$-sets are precisely the definable sets in $\structure(\exp)$.
In particular, all definable functions are piecewise given by terms in an appropriate language.
The proof also shows that an $\LE$-set in $\rReal$ has finitely many connected components, and so $\structure(\exp)$ is o-minimal.

One considers the following kind of $\LE$-sets, whose projection to the first coordinates is, by construction, an $\LE$-set.
\begin{definition}
	Let $C'\subset\rReal^n$ be an $\LE$-set.
	An \emph{$\LE$-cylinder} $C\subset\rReal^{n+1}$ with \emph{base} $C'$ is either the graph of an $\LE$-function $f$ over $C'$ or the region, which we call an \emph{interval}, between two such graphs for functions $f_1,f_2$.
	We also include in this second case the possibility of a region between the graph of an $\LE$-function and $\pm\infty$.
	We write $C=\graph_{C'}(f)$ or $C=(f_1,f_2)_{C'}$, accordingly.

	We call the functions $f,f_1,f_2$ the \emph{walls} of $C$ and, in the case of an interval, require that $f_1$ is strictly smaller than $f_2$ uniformly over $C'$.
\end{definition}
\begin{remark}
	\label{rem: cell vs cylinder}
	We note the following differences between an $\LE$-cylinder and the similar notion of an o-minimal cell.
	First, the base $C'$ of a cylinder $C$ may be disconnected.
	A priori, it might have infinitely many connected components, though this is eventually ruled out by the o-minimality of $\structure(\exp)$.
	Second, the walls of an $\LE$-cylinder are not required to be continuous.
	
	On the other hand, the intersection of two $\LE$-cylinders is a single $\LE$-cylinder --- for example, the intersection of two cylinders of the form $(f_1,g_1)_{C'_1}$ and $(f_2,g_2)_{C'_2}$ has as a base the $\LE$-set given by the intersection of the sets $C'_1$, $C'_2$, $\{f_1<g_2\}$, and $\{f_2<g_1\}$.
	The walls of the resulting cylinder are given by $\max\{f_1,f_2\}$ and $\min\{g_1,g_2\}$.
	We rely on this fact in our constructions since, a priori, the number of connected components in such an intersection might be large with respect to the complexities of the given cylinders (though this is eventually ruled out by our results).
	We also have that the fiber over every base point of $C'$ is a connected subset of $\rReal$.
\end{remark}

\begin{remark}
	\label{rem: refine two covers}
	Since the intersection of two $\LE$-cylinders is a single $\LE$-cylinder, one may take a common refinement (of controlled size) of two coverings of $\rReal^n$ by $\LE$-cylinders by considering all pairwise intersections between elements of the first cover and elements of the second cover.
	In the same way, one may refine a given $\LE$-cylinder in $\rReal^n$ by the elements of a covering of $\rReal^n$ by $\LE$-cylinders.
\end{remark}

A corollary of the main results of~\cite{vdDriesSpeissegger2002} is the following, from which \mbox{o-minimality} of $\structure(\exp)$ and the description of definable sets as $\LE$-sets follow.
\begin{theorem}[{\cite[Theorem 3.1]{vdDriesSpeissegger2002}}]
	\label{thm: LR qualitative}
	Every $\LE$-set is a finite union of $\LE$-cylinders.
\end{theorem}
To establish this, it suffices to show that for any finite collection $f_1,\dots,f_k\in\LE_n$, there is a finite cover of $\rReal^n$ by $\LE$-cylinders on which the functions $f_i$ have constant sign.
We do the same, but with effective control on the number and complexity of $\LE$-cylinders in this cover, and with polynomial dependence on $k$ and on the degrees of the functions $f_1,\dots,f_k$.
Our constructions are similar to those of~\cite{vdDriesSpeissegger2002}, however additional care is needed to ensure polynomial dependence on the number of functions $k$, which is crucial in order for the inductive argument to go through.

\subsection{Complexity of LE-functions}
We wish to obtain an effective version of \Cref{thm: LR qualitative} and hence need to introduce appropriate notions of complexity for $\LE$-functions and for $\LE$-cylinders.

We will measure the complexity of an $\LE$-function $f$ by considering the \emph{parse trees} associated to $f$, which are defined as follows.
	
\begin{definition}
	\label{def: parse tree}
	A \emph{parse tree} for a function $f\in\LE_{n}$ is a (rooted) tree whose structure records the iterated composition of functions in $\structure$, of $\exp$, and of $\log$ which yield $f$.

	An application of a function in $\structure$ corresponds to a node labeled ``$\structure$'', whose child nodes correspond to the inputs of this function.
	Similarly, applications of $\log$ or $\exp$ correspond to a node labeled ``$\log$'' or ``$\exp$'', respectively, whose child node is the input of this function.
	Leaves (i.e.\ childless nodes) are labeled by one of the variables $x_1,\dots,x_n$, which may appear multiple times.

	We will refer to the inner (non-leaf) nodes of this tree as ``$\structure$-nodes'', ``$\exp$-nodes'', or ``$\log$-nodes'', according to their label.
	We will similarly refer to leaves with a given label, e.g.\ ``$x_1$-leaves''.
\end{definition}

\begin{remark}
	\label{rem: contracting edges}
	Since the collection of functions in $\structure$ is closed under composition, we may contract the edge between any two $\structure$-nodes in a parse tree and obtain a tree describing the same $\LE$-function.
	Similarly, we may contract an edge between an $\exp$-node and $\log$-node, replacing in both cases the two nodes by a single $\structure$-node.

	We will always assume that we work with parse trees which are maximally contracted in this sense.
\end{remark}
	
We note that a single $\LE$-function may be described by different parse trees and, abusing notation, we identify from now on between an $\LE$-function and any corresponding (fixed)  parse tree.

\begin{definition}
	\label{def: LE complexity}
	We say a function $f\in\LE_{n}$ has \emph{complexity} $(\format,\degree)$ if its parse tree consists of at most $\format$ nodes, and its $\structure$-nodes correspond to functions in~$\FD[\format][\degree_i]$, where $\sum\degree_i\leq\degree$.
	We require $F\geq n$.
\end{definition}

\begin{definition}
	\label{def: cylinder complexity}
	We say that an $\LE$-cylinder has \emph{complexity} $(\format,\degree)$ if its walls, as well as the indicator function of its base, have complexity $(\format,\degree)$.
	
	In particular, the graph of an $\LE$-function $f:\rReal^n\to\rReal$ of complexity $(\format,\degree)$, considered as an $\LE$-cylinder with base $\rReal^n$, has complexity $(\format,\degree)$.
\end{definition}

\begin{lemma}
	\label{lem: complexity to indicator}
	Let $C\subset\rReal^{n+1}$ be an $\LE$-cylinder of complexity $(F,D)$.
	Then the indicator function of $C$ is an element of $\LE_{n+1}$ which may be taken to be of complexity $(\order[\format]{1},\polyfd)$.
\end{lemma}
\begin{proof}
	Let $C'\subset\rReal^n$ be the base of $C$.
	We treat the case where $C=(f_1,f_2)_{C'}$ for $f_1,f_2\in\LE_n$, the other cases being similar.
	By \Cref{def: cylinder complexity}, the indicator function $\chi_{C'}$ of $C'$ and the walls $f_1,f_2$ have complexity $(\format,\degree)$.
	Let $\chi_{+}$ be the indicator function of the positive reals.

	We may write the indicator function of $C$ as 
	\begin{equation}
		\chi_{C'}(\initial{x}{1}{n})\cdot\chi_{+}(x_{n+1}-f_1(\initial{x}{1}{n}))\cdot
			\chi_{+}(f_2(\initial{x}{1}{n})-x_{n+1}),
	\end{equation}
	from which the claim follows directly, by \Cref{def: LE complexity}.
\end{proof}

The complexity pairs of \Cref{def: LE complexity,def: cylinder complexity} form a double filtration on the collections of $\LE$-functions and of $\LE$-cylinders.
They are analogous to format and degree and will be used in \Cref{sec: so min} in order to define an FD-filtration on the definable sets of $\structure(\exp)$ and show that it is (up to equivalence of FD-filtrations) sharply o-minimal with sharp cell decomposition.

We stress, however, that the components $\format,\degree$ in \Cref{def: cylinder complexity} above are not themselves the same as the format and degree of the corresponding $\LE$-cylinder in the FD-filtration we construct in \Cref{sec: so min}.	
In order to avoid confusion, in the context of \Cref{def: LE complexity,def: cylinder complexity} we will refer to $\format$ and $\degree$ as \emph{tree-format} and \emph{tree-degree} to distinguish them from the format and degree of an FD-filtration.

Our effective version of \Cref{thm: LR qualitative} is the following.
\begin{theorem}
	\label{thm: LE cylinder decomposition}
	Let $f_1,\dots,f_k\in\LE_{n+1}$.
	Then there is a decomposition of $\rReal^{n+1}$ into $\LE$-cylinders, on each of which the functions $f_i$ have constant sign.
	
	If $f_1,\dots,f_k$ are of complexity $(F,D)$, we may take the number of cylinders in this decomposition to be $\polyfd[\format][\degree,k]$ and the complexity of each of the resulting cylinders to be $(\order[\format]{1},\polyfd[\format][\degree])$.
\end{theorem}
This theorem will follow from more precise preparation theorems which we give in \Cref{sec: prepartion theorems}.

\subsection{Sharp o-minimality and sharp cell decomposition}
\label{sec: so min}
In this subsection, we prove \Cref{thm: S exp so min} using  \Cref{thm: LE cylinder decomposition}.
Already from \Cref{thm: LR qualitative}, we have that every definable set in $\structure(\exp)$ is a finite union of $\LE$-cylinders.
We may thus define the following FD-filtration on $\structure(\exp)$.
\begin{definition}
	\label{def: set complexity}
	Let $A\subset \rReal^n$ be in $\structure(\exp)$ and let $\format\geq n$.
	We will say that $A\in\structure(\exp)_{\format,\degree}$ if it is the union of $\polyfd[\format][\degree]$-many $\LE$-cylinders of complexity $(\order[F]{1},\polyfd[\format][\degree])$.
\end{definition}
The precise choices of asymptotic bounds in this definition will be determined by the proofs of the results of this section.
A simple example of this is the following.

\begin{proposition}
	\label{prop: complexity translation}
	Let $A\subset\rReal^{n}$ be an $\LE$-set whose indicator function, as an $\LE$-function, is of complexity $(\format,\degree)$.
	Then $A\in\structure(\exp)_{\order[\format]{1},\polyfd}$.
\end{proposition}
\begin{proof}
	Applying \Cref{thm: LE cylinder decomposition} to the indicator function of $A$, we obtain $\polyfd$-many $\LE$-cylinders of complexity $(\order[\format]{1},\polyfd)$ whose union is $A$.
	We may choose the asymptotic bounds in \Cref{def: set complexity} large enough such that the result follows.
\end{proof}
Note that we make no claim in the converse direction, that is an $\LE$-cylinder in $\structure(\exp)_{\format,\degree}$ might have large tree-format and tree-degree (compare with \Cref{lem: complexity to indicator}, where we assume an $\LE$-cylinder has low tree-format and tree-degree and deduce the same for its indicator).

We restate \Cref{thm: S exp so min} using the filtration defined in \Cref{def: set complexity}:
\begin{theorem}
	The FD-filtration $\{\structure(\exp)_{\format,\degree}\}$ is equivalent to a sharply o-minimal filtration with sharp cell decomposition.
\end{theorem}
By~\cite[Proposition 1.18]{BinyaminiNovikovZak2022}, it is enough to show that $\{\structure(\exp)_{\format,\degree}\}$ admits sharp cell decomposition and is \emph{weak-sharply o-minimal} in the sense of~\cite[Definition 1.7]{BinyaminiNovikovZak2022} --- that is, we allow projections, complements, and finite intersections to increase the format by $1$.
We begin with weak-sharp o-minimality.

\begin{theorem}
	\label{thm: s exp sharp}
	The filtration $\{\structure(\exp)_{\format,\degree}\}$ is weak-sharply o-minimal.
\end{theorem}
\begin{proof}

	We first bound the number of connected components of a definable subset of $\rReal$.
	Let $A\subset \rReal$ be in $\structure(\exp)_{\format,\degree}$.
	Since an $\LE$-cylinder in $\rReal$ is connected, \Cref{def: set complexity} implies that~$A$ has at most $\polyfd$ connected components.

	We now show that $\{\structure(\exp)_{\format,\degree}\}$ is weak-sharply generated.
	Let $P\in\rReal[x_1,\dots,x_n]$ be of degree $d$.
	Since any function in $\structure$ is also an $\LE$-function, a semialgebraic cell is, in particular, an $\LE$-cylinder.
	Thus the existence of an effective semialgebraic cell decomposition of $\rReal^n$ compatible with $P$ (for example, using the algebraic CPT~\cite[Theorem 8]{BinyaminiNovikov2019}) yields that $\{P=0\}$ is in $\structure(\exp)_{n,d}$ upon choosing suitable asymptotic bounds in \Cref{def: set complexity}.

	Let $A\subset\rReal^{n}$ be in $\structure(\exp)_{\format,\degree}$.
	We bound the format and degree of $A\times \rReal$ and $\rReal\times A$.
	Write $A$ as the union of $\polyfd$-many $\LE$-cylinders of complexity $(\order[\format]{1},\polyfd)$.
	By \Cref{lem: complexity to indicator}, one may choose for each of these cylinders an indicator function which is an $\LE$-function of complexity $(\order[\format]{1},\polyfd)$.
	We may thus take the dependence on $\format$ in the asymptotic bounds in \Cref{def: set complexity} to be such that $A\times\rReal\in \structure(\exp)_{\format+1,\degree}$, and similarly for $\rReal\times A$.

	Let $A_1,\dots,A_k\subset\rReal^{n+1}$ where $A_i\in\structure(\exp)_{\format,\degree_i}$.
	Let $D=\sum D_i$.
	It is immediate from \Cref{def: set complexity} that $\bigcup A_i\in \structure(\exp)_{\format,\degree}$ since $\sum D_i^\alpha\leq (\sum D_i)^\alpha$ for any $D_1,\dots,D_k,\alpha\geq 1$.

	It remains to estimate the format and degree of the intersection $\bigcap A_i$, of the projections $\pi(A_i)$ of each $A_i$ to $\rReal^n$, and of the complements $\rReal^{n+1}\setminus A_i$ for each $A_i$.
	Write each $A_i$ as the union of $\polyfd[\format][\degree_i]$-many $\LE$-cylinders of complexity $(\order[\format]{1},\polyfd[\format][\degree_i])$.
	Let $\{f_1,\dots, f_N\}\subset\LE_n$ be the collection of $N=\polyfd[\format][\degree,k]$ functions determining these cylinders --- their walls and the indicator functions of their bases.
	We may assume without loss of generality that $f_1\equiv 0$.
	Apply \Cref{thm: LE cylinder decomposition} to the collection of all pairwise differences of elements in $\{f_1,\dots, f_N\}$ and let $\mathcal{A}$ be the resulting collection of $\LE$-cylinders covering $\rReal^n$.

	The cylinders in $\mathcal{A}$ are compatible with the projections $\pi(A_i)$ since they are compatible with the indicator functions of the bases of the cylinders composing each of the $A_i$.
	Over each cylinder in $\mathcal{A}$, the functions $\{f_1,\dots, f_N\}$ are uniformly and strictly ordered, and so they determine $\LE$-cylinders in $\rReal^{n+1}$ which are compatible with each $A_i$.
	In particular, these $\LE$-cylinders are compatible with the intersection $\bigcap A_i$ and with the complements $\rReal^{n+1}\setminus A_i$.
	We may thus choose the dependence on $\format$ in the asymptotic bounds in \Cref{def: set complexity} to be such that $\pi(A_i)$, $\rReal^{n+1}\setminus A_i$, and $\bigcap A_i$ are in $\structure(\exp)_{\format+1,\degree}$.
\end{proof}

We now prove sharp cell decomposition.
For later convenience, we do this with respect to the following kind of cells.
\begin{definition}
	\label{def:LE cell}
	Let $C\subset\rReal^{n}$ be an $\LE$-cylinder.
	We say that it is an \emph{$\LE$-cell} if either $n=0$ or else if its base is an $\LE$-cell and its walls are continuous (when restricted to this base).
	An $\LE$-cell will be called \emph{admissible} if it is compatible with all coordinate hyperplanes.
\end{definition}

\begin{theorem}
	\label{thm: sharp cd}
	The FD-filtration $\{\structure(\exp)_{\format,\degree}\}$ has sharp cell decomposition.
	More precisely, for $A_1,\dots,A_k\subset\rReal^{n+1}$ which are in $\structure(\exp)_{\format,\degree}$, there exists a decomposition of $\rReal^{n+1}$ into $\polyfd[\format][\degree,k]$ cells in $\structure(\exp)_{\order[\format]{1},\polyfd[\format][\degree]}$ which are compatible with each~$A_i$.
	We may take these cells to be admissible $\LE$-cells.
\end{theorem}
\begin{proof}
	Since $\format\geq n$ by \Cref{def: set complexity}, it suffices to prove the claim with all instances of $\poly_{\format}$ and $O_{\format}$ replaced with $\poly_{\format,n}$ and $O_{\format,n}$, respectively.
	We proceed similarly to the proof of \Cref{thm: s exp sharp} and by induction on~$n$.
	By \Cref{def: set complexity}, each of the sets $A_i$ is a union of $\polyfd$-many $\LE$-cylinders of complexity $(\order[\format]{1},\polyfd)$.
	
	Let $\{f_1,\dots,f_N\}\subset\LE_n$ be the collection of $N=\polyfd[\format][\degree,k]$ functions determining this collection of cylinders --- their walls and the indicator functions of their bases.
	We may assume without loss of generality that $f_1\equiv 0$.
	Applying \Cref{thm: LE cylinder decomposition} to all pairwise differences of functions in $\{f_1,\dots,f_N\}$, we obtain a decomposition of $\rReal^n$ into $\LE$-cylinders.

	Let $\mathcal{A}$ be the collection of subsets of $\rReal^n$ consisting of these cylinders and of the discontinuity sets of the functions $f_1,\dots,f_N$.
	Since $\{\structure(\exp)_{\format,\degree}\}$ is weak-sharply o-minimal (by \Cref{thm: s exp sharp}), these discontinuity sets are all in $\structure(\exp)_{(\order[\format]{1},\polyfd)}$ (using the standard $\epsilon$--$\delta$ formulation).
	
	We inductively apply the theorem to the sets in $\mathcal{A}$, obtaining a decomposition of $\rReal^n$ into admissible $\LE$-cells which is of suitable size and complexity.
	We note that, by o-minimality of $\structure(\exp)$, cells contained in a discontinuity set of one of the functions $f_i$ must be of dimension less than $n$.

	Over each of the resulting cells, the walls of the cylinders composing the sets $A_i$ are uniformly and strictly ordered, and compatible with $0$.
	Over cells of dimension $n$, these walls are continuous and so determine admissible $\LE$-cells.
	Denote this collection of admissible $\LE$-cells by $\mathcal{D}_1$.
	
	Over cells of lower dimension, we proceed by induction on $n$ in the following way.
	Let $C\subset\rReal^n$ be one of the resulting admissible $\LE$-cells which is of dimension less than~$n$.
	For simplicity, we assume that $C$ is the graph of an $\LE$-function $g$ over a base $C'\subset\rReal^{n-1}$, the other cases are treated similarly.
	For each of the functions $f_i$, let $\tilde f_i:C'\to\rReal$ be given by 
	\begin{equation}
		\tilde f_i(x_1,\dots,x_{n-1})=f_i (x_1,\dots,x_{n-1},g(x_1,\dots,x_{n-1})).
	\end{equation}
	The functions $\{\tilde f_1,\dots,\tilde f_N\}$ determine pullbacks of the cylinders composing the sets $\mathcal{A}_i$ over $C$ to cylinders over $C'$.
	We may now apply the theorem inductively to these cylinders, which are subsets of $\rReal^n$.
	Let $\{D_i\}$ be the resulting collection of admissible $\LE$-cells in $\rReal^n$.
	For each $D_i$, let $\tilde D_i\subset\rReal^{n+1}$ be the admissible $\LE$-cell given by 
	\begin{equation}
		\{(x_1,\dots ,x_{n-1},g(x_1,\dots,x_{n-1}),x_n)\st (x_1,\dots,x_n)\in D_i\}.
	\end{equation}
	The collection of $\LE$-cells $\mathcal{D}_2=\{\tilde D_i\}$, together with the collection $\mathcal{D}_1$, form the desired decomposition of $\rReal^{n+1}$ into admissible $\LE$-cells.
\end{proof}

\begin{remark}
	Using \Cref{thm: sharp cd}, we may assume that all applications of the $\LE$-preparation and $\LA$-preparation theorems (\Cref{thm: exp split,thm: LS preparation}), and hence also of \Cref{thm: LE cylinder decomposition}, yield decompositions into admissible $\LE$-cells rather than just $\LE$-cylinders.
	We note that the inductive proofs of these theorems rely on the use of cylinders, and so we only make this assumption starting from \Cref{sec: complex cells}.
\end{remark}

\subsection{Effective bounds on the lengths of existential formulas}
\label{sec: effective model completeness}
\Cref{thm: LE cylinder decomposition} yields, in particular, an effective theorem of the complement (i.e.\ effective model completeness) for existentially defined sets in $\structure(\exp)$ with respect to the language $\langle\rReal,\exp,\mathcal{F}\rangle$, where $\mathcal{F}$ is the collection of all functions definable in~$\structure$.

In the case of $\structure=\Rre$, we have \Cref{thm: effective model completeness} below, which is closer in spirit to Wilkie's model completeness results in \cite{Wilkie1996} (see \cite{BerarducciServi2004} for an effectivization of Wilkie's theorem of the complement of \cite{Wilkie1999}).

We work in the langauge $\mathcal{L}=\langle \rReal,+,\cdot,\exp,\restrict{\sin}{[0,1]} \rangle$.
We note that our use of~$\restrict{\sin}{[0,1]}$ is necessary due to our complex analytic methods.

\begin{definition}
	\label{def: quantifier free}
	A set $X\subset\rReal^n$ is a \emph{basic quantifier-free set in $\mathcal{L}$} if it is of the form
	\begin{equation}
		\label{eq: qf sign set}
		\{f_1=\cdots=f_r=0,g_1>0,\dots,g_s>0\}
	\end{equation}
	for some $r,s\in\nNatural$, where each function $f_i$ or $g_i$ is of the form 
	\begin{equation}
		\label{eq: exponential trig polynomial}
		P_{i}(x_1,\dots,x_n,\exp(x_1),\dots,\exp(x_n),\restrict{\sin}{[0,1]}(x_1),\dots,\restrict{\sin}{[0,1]}(x_n))
	\end{equation}
	for a polynomial $P_{i}\in\rReal[t_1,\dots,t_{3n}]$.
	
	A set $X\subset\rReal^n$ is a \emph{quantifier-free set in $\mathcal{L}$} if it is a finite union of basic quantifier-free sets.
	The \emph{degree} of $X$ is the sum of the degrees of all polynomials $P_i$ used to define the functions $f_i,g_i$ determining $X$.
\end{definition}
Note that we only restrict the domain of $\sin$ and not of $\exp$ in \eqref{eq: exponential trig polynomial}.

\begin{theorem}
	\label{thm: effective model completeness}
	Let $\widehat{X}\subset\rReal^{n+m}$ be a quantifier-free set in $\mathcal{L}$ of degree $D$ and let $X\subset\rReal^n$ be the projection of $\widehat{X}$ to $\rReal^n$.
	Then there exist $N=\poly_{n,m}(D)$ and a quantifier-free set $\widehat{Y}\subset \rReal^{N}$ of degree $\poly_{n,m}(D)$ such that its projection $Y\subset\rReal^n$ satisfies $\rReal^n\setminus X = Y$.
\end{theorem}
\begin{proof}
	The functions $f_i,g_i$ determining $\widehat{X}$ in \eqref{eq: qf sign set} are $\LE$-functions of complexity $(\order[n,m]{1},\polyfd[n,m][D])$.
	Applying \Cref{thm: LE cylinder decomposition} to the functions $f_i,g_i$ yields a decomposition of $\rReal^n$ into $\polyfd[n,m][D]$-many $\LE$-cylinders of complexity $(\order[n,m]{1},\polyfd[n,m][D])$.

	Let $M=\polyfd[n,m][D]$ and let $h_1,\dots,h_M\in\Rre$ be the functions corresponding to all $\structure$-nodes in the parse trees of the $\LE$-functions determining the resulting cylinders.
	Each $h_i$ is in $\FD*[n,m][D]$, where we consider the structure $\structure=\Rre$ with the FD-filtration defined in~\cite[Definition 2.1]{CarmonAnalyticallyGenerated}.
	By~\cite[Proposition 2.5]{CarmonAnalyticallyGenerated}, this filtration is reducible to the \emph{$*$-filtration} given in~\cite[Definition 3]{BinyaminiVorobjov2022}.
	Thus each $h_i$ has \emph{$*$-format} $\order[n,m]{1}$ and \emph{$*$-degree} $\polyfd[n,m][D]$.
	
	We may consider the functions $h_i$ to be defined on a common space $\rReal^{\order[n,m]{1}}$.
	Applying \cite[Theorem 1]{BinyaminiVorobjov2022} to the functions $h_i$, we obtain a cell decomposition\footnote{While \cite[Theorem 1]{BinyaminiVorobjov2022} is formulated in terms of sub-Pfaffian cells, an inspection of the proof shows that the resulting cells are projections of quantifier-free sets in $\mathcal{L}$.
	That is, they are sub-Pfaffian with respect to the same Pfaffian chain as the functions $h_i$.}
	of $\rReal^{\order[n,m]{1}}$ by cells of $*$-format $\order[n,m]{1}$ and $*$-degree $\polyfd[n,m][D]$.
	By the cellular decomposition result of Gabrielov and Vorobjov in \cite{GabrielovVorobjov2001}, these cells are given by projections of quantifier-free sets in $\rReal^N$ of degree $\polyfd[n,m][D]$, for $N={\polyfd[n,m][D]}$.
	
	Combining this existential description of the functions $h_i$ and the parse trees of the $\LE$-functions determining the $\LE$-cylinders covering $\rReal^{n+m}$ yields a cylindrical decomposition of $\rReal^{n+m}$ by projections of quantifier-free sets in~$\rReal^N$ which are compatible with $\widehat{X}$.
	Taking $\widehat{Y}$ to be the union of those cylinders whose projection to $\rReal^n$ does not intersect $X$ finishes the proof.
\end{proof}

\section{Preparation theorems}
\label{sec: prepartion theorems}
In this section, we prove \Cref{thm: LE cylinder decomposition}.
Ignoring questions of effectivity, this section mostly follows~\cite{vdDriesSpeissegger2002}.
We have simplified some parts of the arguments of loc.\ cit., while other parts require more care.

We introduce here the following terminology, which will be convenient throughout the rest of the paper.
\begin{definition}
	\label{def: unit}
	Let $S$ be a set and let $\mu>0$.
	A \emph{unit} on $S$ is a positive function $u:S\to\rReal_{>0}$, which is bounded away from $0$ and from infinity.
	
	A \emph{$\mu$-unit} on $S$ is a unit $u$ such that its logarithmic width (i.e.\ the diameter of $\log(u(S))\subset\rReal$) is at most $\mu$. 
\end{definition}

We fix $n$ for the rest of this section and denote from now on the coordinates of~$\rReal^{n+1}$ by $x_1,\dots,x_n,y$.
We will treat separately the following kind of $\LE$-functions, called ``purely logarithmic in the last variable'' in~\cite{vdDriesSpeissegger2002}
(see \Cref{def: parse tree} for the notions of $\exp$-nodes, $\log$-nodes, and $y$-leaves). 
\begin{definition}
	\label{def: LA function}
	A function $f\in\LE_{n+1}$ is an \emph{$\LA$-function} if there exists a parse tree for $f$ with no $\exp$-nodes on any path connecting the root to a $y$-leaf.
	If, in addition, all paths connecting the root to a $y$-leaf in this tree contain at most $\ell$ $\log$-nodes, we write~$f\in\LA_{n+1,\ell}$ and call $\ell$ the \emph{logarithmic depth} of $y$ in~$f$, or simply the logarithmic depth of $f$, for short.
\end{definition}
When considering the parse tree of a function in $\LA_{n+1,\ell}$, we will always assume that it is of the form in the definition above.

\Cref{thm: LE cylinder decomposition} reduces to the following two claims.

\begin{proposition}
	\label{prop: LE to LS}
	Let $f_1,\dots,f_k\in\LE_{n+1}$.
	Then there is a finite decomposition of $\rReal^{n+1}$ into $\LE$-cylinders $\{C_j\}$ along with corresponding $\LA$-functions $g_{i,j}:\rReal^{n+1}\to\rReal$ such that $\sign f_i=\sign g_{i,j}$ on $C_j$.

	If the functions $f_1,\dots,f_k$ are of complexity $(\format,\degree)$, we may take the number of cylinders in this decomposition to be $\polyfd[\format,n][\degree,k]$ and the complexities of each $C_j$ and $g_{i,j}$ to be $(\order[\format,n]{1},\polyfd[\format,n][\degree])$.
\end{proposition}
\begin{proposition}
	\label{prop: LS constant sign}
	Let $f_1,\dots,f_k\in\LA_{n+1}$.
	Then there is a finite decomposition of $\rReal^{n+1}$ into $\LE$-cylinders, on each of which the functions $f_i$ have constant sign.

	If the functions $f_1,\dots,f_k$ have complexity $(F,D)$, we may take the number of cylinders in this decomposition to be $\polyfd[\format,n][\degree,k]$ and the resulting cylinders to be of complexity $(\order[F,n]{1},\polyfd[\format,n][\degree])$.
\end{proposition}

\Cref{prop: LE to LS} follows from the more explicit statement given in \Cref{thm: exp split} below (see \Cref{def: syntactic complexity} for the notion of exponential level).
\begin{definition}[$\LE$-prepared]
	\label{def: LE prepared}
	Let $C\subset\rReal^{n+1}$ be an $\LE$-cylinder and let $\mu>0$.
	A function $f:C\to\rReal$ in~$\LE_{n+1}$ is \emph{$(\LE,\mu)$-prepared} if there exists a function $g\in\LA_{n+1}$, a function $h$ in $\LE_{n+1}$ of lower exponential level than $f$, and a $\mu$-unit $u\in\LE_{n+1}$ such that
	\begin{equation}
		\label{eq: LA times exp}
		f=g\cdot\exp(h)\cdot u
	\end{equation}
	on $C$.
	When it is clear from context that $f$ is an $\LE$-function, we will simply say that it is \emph{$\mu$-prepared}.
\end{definition}
\begin{theorem}[$\LE$-preparation]
	\label{thm: exp split}
	Let $f_1,\dots,f_k\in\LE_{n+1}$ be of complexity $(\format,\degree)$ and let $\mu>0$.
	Then there is a decomposition of $\rReal^{n+1}$ into $\polyfd[\format,n][\degree,k,1/\mu]$-many $\LE$-cylinders, on each of which each $f_i$ is $\mu$-prepared.
	
	We may take the complexities of the cylinders in this decomposition and of the functions $g_{i},h_{i},u_{i}$ in \eqref{eq: LA times exp} to be $(\order[F,n]{1},\polyfd[\format,n][\degree])$.
\end{theorem}

Similarly, \Cref{prop: LS constant sign} follows from a more explicit ``$\LA$-preparation theorem'' (\Cref{thm: LS preparation}) which we defer to \Cref{sec: LS preparation proof} after some additional technical setup.

We first prove \Cref{thm: exp split} in \Cref{sec: LE preparation proof} under the assumption that \Cref{prop: LS constant sign} holds, and then proceed to prove \Cref{prop: LS constant sign} in \Cref{sec: LS preparation proof}.

\begin{remark}
	While \Cref{def: LE prepared} suffices for the proof of \Cref{thm: LE cylinder decomposition}, see \Cref{def: complex cell prepared,thm: single cell} in \Cref{sec: single cell} for our full preparation theorem for $\LE$-functions using complex cells.
\end{remark}

The main tools that we will use in the proofs of \Cref{thm: exp split,thm: LS preparation} are the following definition and corollary of~\cite[Theorem 5.3]{CarmonAnalyticallyGenerated}.

\begin{definition}[$\structure$-prepared]
	\label{def: mu prepared}
	Let $C\subset\rReal^m$ be an o-minimal cell in $\structure$ with base $C'\subset\rReal^{m-1}$.
	Assume that $C$ is compatible with $\{x_m=0\}$.
	Let $\mu>0$.
	A function $f:C\to\rReal$ is \emph{$(\structure$,$\mu)$-prepared} if there exist definable functions $\theta,a:C'\to\rReal$ which are compatible with $0$, a definable $\mu$-unit $u$ on $C$, and an exponent $\lambda\in\qRationals$, satisfying the following two conditions.
	
	First, we have that $f$ is given by
	\begin{equation}
		\label{eq: mu prepared}
		f(\initial{\varX}{1}{m})=\abs{x_m-\theta(\initial{\varX}{1}{m-1})}^{\lambda}\cdot a(\initial{\varX}{1}{m-1})\cdot u(\initial{\varX}{1}{m})
	\end{equation}
	on $C$.

	Second, on those cells where $x_m$ and $\theta$ are both non-zero, we have
	\begin{equation}
		\label{eq: center condition}
		0<\abs{x_m-\theta}<\frac{1}{2}\abs{x_m}.
	\end{equation}

	When $\structure$ is clear from context, we will simply say that $f$ is \emph{$\mu$-prepared}.
	We will also call the function $\theta$ the \emph{center} of the expansions~\eqref{eq: mu prepared}.
\end{definition}
We note that the functions $a,\theta$ in \Cref{def: mu prepared} above depend only on the first $m-1$ coordinates, while $u$ may depend on all $m$ coordinates.

\begin{theorem}[$\structure$-preparation]
	\label{thm: S-preparation}
	Let $f_1,\dots,f_k:\rReal^m\to\rReal$ be in $\FD[\format][\degree]$.
	Let $\mu>0$.
	Then there is a decomposition of $\rReal^m$ into $\polyfd[\format][\degree,k,\mu^{-1}]$ o-minimal cells in $\FD*[\format][\degree]$ on each of which each $f_i$ is $\mu$-prepared with respect to a common center $\theta$.

	We may take the exponents $\lambda_i\in\qRationals$ in \eqref{eq: mu prepared} to be of height at most $\polyfd$, and the functions $\theta,a_i,u_i$ in \eqref{eq: mu prepared} to be in $\FD*[\format][\degree]$.
\end{theorem}
We note that the center $\theta$ in \Cref{thm: S-preparation} above depends on the choice of cell but not on $i$, that is we use the same center $\theta_i=\theta$ for all functions $f_i$.

\begin{remark}[Covers versus decompositions]
	While~\cite[Theorem 5.3]{CarmonAnalyticallyGenerated} yields a cover of $\rReal^m$ by cells rather than a disjoint decomposition, one can obtain a decomposition in either of the following two ways.
	First, we may replace the use of the real \s CPT in the proof of~\cite[Theorem 5.3]{CarmonAnalyticallyGenerated} by its ``disjoint version'' (see~\cite{Shankar2025} and \cite[Remark A.4]{BinyaminiCarmonNovikovComplexCells}).
	Alternatively, it suffices for our purposes to refine the resulting cover of~\cite[Theorem 5.3]{CarmonAnalyticallyGenerated} by disjoint cells of slightly higher complexity, by invoking our assumption that the structure $\{\FD[\format][\degree]\}$ has sharp cell decomposition.
\end{remark}

\subsection{\texorpdfstring{$\LE$-preparation}{LE-preparation}}
\label{sec: LE preparation proof}
In this section, we prove \Cref{thm: exp split}.
The proof proceeds by induction on several combinatorial features of the parse trees of the functions $f_1,\dots,f_k$ in the statement of the theorem.
We now describe these features.

\begin{definition}[Syntactic complexity]
	\label{def: syntactic complexity}
	Let $f\in\LE_{n+1}$.
	The \emph{exponential level} of $f$, which we denote by $e(f)$, is the maximal number of $\exp$-nodes on a path connecting a $y$-leaf to the root of the tree (including the root, if it is itself an $\exp$-node).

	A \emph{maximal subtree} of $f$ is a subtree $g$, rooted at an $\exp$-node, such that $e(f)=e(g)$.
	We denote by $\mathcal{E}(f)$ the set of distinct maximal subtrees of $f$.

	The \emph{logarithmic depth} $\operatorname{ld}(g)$ of a subtree $g$ of $f$ is the maximal number of $\log$-nodes on a path connecting the root of $f$ with any appearance of~$g$ (including the roots of $f$ and $g$, if they are themselves $\log$-nodes).
	Write $\ell(f)$ for the maximum of $\operatorname{ld}(g)$ over all maximal subtrees $g$ of~$f$.

	We denote by $\mathcal{LE}(f)\subset\mathcal{E}(f)$ the set of maximal subtrees that attain this maximal number of $\log$-nodes on a path connecting the root of $f$ with at least one of their occurrences in $f$.

	The \emph{syntactic complexity} of $f$ (called ``logarithmic--exponential nestedness'' in~\cite{vdDriesSpeissegger2002}) is the tuple 
	\begin{equation}
		(e(f),\#{\mathcal{E}(f)},\ell(f),\#{\mathcal{LE}(f)})\in\zIntegers_{\geq 0}^4.
	\end{equation}
	We order these tuples lexicographically, where $e(f)$ is the most significant component and $\#{\mathcal{LE}(f)}$ is the least significant.
\end{definition}

We note that all the terms in the syntactic complexity of an $\LE$-function are bounded by its tree-format.
Hence it suffices to prove \Cref{thm: exp split} with all asymptotic dependencies on $F$ replaced by dependence on the tuple $(F,\widetilde{F})$, where $\widetilde{F}=(e,\#{\mathcal{E}},\ell,\#{\mathcal{LE}})\in\zIntegers_{\geq 0}^4$ is an upper bound (with respect to the lexicographic order) for the syntactic complexities of the functions $f_1,\dots,f_k$ in the statement of the theorem.
We assume inductively that \Cref{thm: LE cylinder decomposition,prop: LE to LS,prop: LS constant sign,thm: exp split} hold for functions of syntactic complexity lower than $\widetilde{F}$.
In particular, we may assume that $\widetilde{F}$ is equal to the syntactic complexity of one of the functions $f_i$.
The base case of the induction is \Cref{prop: LS constant sign}, which corresponds to $e=0$.
Hence we may also assume from now on that $e>0$.

Throughout the proof, we will consider the restrictions of the functions $f_i$ to various $\LE$-cylinders.
Starting with such a cylinder $C$, we will often refine it by decomposing it into $\LE$-cylinders on which the functions $f_i$ are equal to $\LE$-functions $\widetilde{f}_i$ which are either simpler or otherwise of some desired form.
We will identify between $f_i$ and $\widetilde{f_i}$ whenever it is understood that we are restricting our attention to an $\LE$-cylinder on which they are equal.

Denote the distinct elements of $\mathcal{LE}(f_i)$ by $\exp(s_{i,1}),\dots,\exp(s_{i,q_i})$, ordered using any (universally fixed) ordering on the collection of $\LE$-functions.
Let $q=\max_i q_i$ (that is, $q=\#\mathcal{LE}$).
Allowing repetitions, we may write $\mathcal{LE}(f_i)=\{\exp(s_{i,1}),\dots,\exp(s_{i,q})\}$.
To ease notation, we abbreviate and write $\exp(s_i)$ for this ordered list of $\LE$-functions.
We will also write $\exp(s'_i)$ for the tuple $(\exp(s_{i,1}),\dots,\exp(s_{i,q-1}))$ which omits the last term $\exp(s_{i,q})$.
We note that these $\LE$-functions  are all of complexity $(F,D)$, as they are given by subtrees of the functions $f_i$.

Consider some occurrence of $\exp(s_{i,q})$ in $f_i$.
We can assume the root of $f_i$ is an $\structure$-node by composing with the identity function.
By contracting all \mbox{$\log$--$\exp$} edges and \mbox{$\structure$--$\structure$} edges in $f_i$, we may assume one of the following holds.
\begin{enumerate}
	\item This occurrence of $\exp(s_{i,q})$ is a child-node of an $\structure$-node which is the root of~$f_i$.
	\item Otherwise, this occurrence of $\exp(s_{i,q})$ is a child-node of an $\structure$-node which is itself the child-node of a $\log$-node.
\end{enumerate}
In either case, denote by $g_{i,j}$ the subtree of $f_i$ rooted at this $\structure$-node parent of $\exp(s_{i,q})$.
The index $j$ accounts for the possible different occurrences of $\exp(s_{i,q})$ as a subtree of $f_i$, each with a different $\structure$-node parent.
In order to make the notation less cumbersome, we will suppress from now on the subscript $j$ and denote these functions as $g_i$.
For each $i=1,\dots,k$, there are at most $\order[\format]{1}$ corresponding values of $j$, and so this does not affect any of our estimates.

There exist $p=\order[\format]{1}$ and $\LE$-functions $r_{i,1},\dots,r_{i,p}$ whose exponential level is lower than $e(f_i)$, such that 
\begin{equation}
	\label{eq: G functions}
	g_{i}=G_{i}(r_{i},\exp(s_i)),
\end{equation}
for $G_{i}\in\FD*[\format][\degree]$ and $r_{i}=(r_{i,1},\dots,r_{i,p})$.

We consider the space $\rReal^{p+q}$ with coordinates $(z,w)=(z_1,\dots,z_p,w_1,\dots,w_q)$.
We also write $w'=(w_1,\dots,w_{q-1})$.
Let the functions $T_i:\rReal^{n+1}\to\rReal^{p+q}$ and~$T'_i:\rReal^{n+1}\to\rReal^{p+q-1}$ be given by
\begin{equation}
	T_i(x,y)=(r_i,\exp(s_i)),\qquad
	T'_i(x,y)=(r_i,\exp(s'_i)).
\end{equation}
Thus, the maps $g_{i}$ factor as $G_{i}\comp T_i$.

We apply $\structure$-preparation (\Cref{thm: S-preparation}) to $\rReal^{p+q}$ and to the $\order[\format]{k}$ functions~$\{G_{i}\}$, obtaining a decomposition of $\rReal^{p+q}$ into $\polyfd[\format][\degree,k,1/\mu]$ cells of suitable complexity, over which each of the functions $G_{i}$ is $\mu/2$-prepared in the sense of \Cref{def: mu prepared}.

The proof now proceeds in two stages.
First, in \Cref{sec: LE pullback of a cell}, we decompose $\rReal^{n+1}$ using a collection of $\LE$-cylinders, each of them mapped by each of the maps $T_i$ into one of the cells covering $\rReal^{p+q}$.
Then, restricting our attention to one of these cylinders, we will further refine it in \Cref{sec: LE pulling back expansions} by $\LE$-cylinders on which each $f_i$ is of the desired form \eqref{eq: LA times exp}.

\subsubsection{Pullback of a cell}
\label{sec: LE pullback of a cell}
Let $\mathcal{D}$ denote the collection of cells covering $\rReal^{p+q}$, obtained by the application of $\structure$-preparation to the functions $\{G_{i}\}$ given in~\eqref{eq: G functions}.
For each cell $D\in\mathcal{D}$, write $D'$ for its base.
Let $\chi_{D'}$ be the indicator function of $D'$.
The cell $D$ is determined over $D'$ by the sign of up to two functions of the form $w_q-\alpha(z,w')$, where $\alpha:D'\to\rReal$ is in $\FD$.
We may assume $\alpha$ is identically $0$ outside of $D'$.
We note that~$\chi_{D'}$ and $\alpha$ are in $\FD*[\format][\degree]$.

We apply \Cref{thm: LE cylinder decomposition} inductively to the following functions, which are of syntactic complexity lower than $\widetilde{F}$ (indeed, either $q=1$, in which case their exponential level is lower than $e$, or else they have the same exponential level but at most $q-1<\#\mathcal{E}$ maximal subtrees):
\begin{equation}
	\label{eq: pullback of a cell LE}
		\chi_{D'}\comp T'_i,\quad \alpha\comp T'_i,\quad s_{i,q}-\log(\alpha\comp T'_i).
\end{equation}
We note that we apply \Cref{thm: LE cylinder decomposition} once, simultaneously, for all values of $D'$, $\alpha$ and~$i$.
(Recall that we extend the value of $\log$ to be zero for all non-positive inputs, so all functions in \eqref{eq: pullback of a cell LE} are well defined).

This yields a collection of $\LE$-cylinders of size $\polyfd[\format,\widetilde{\format},n][\degree,k,1/\mu]$, each of them of complexity $(\order[F,\widetilde{\format},n]{1},\polyfd[\format,\widetilde{\format},n][\degree])$.
Let $C$ be one of the resulting $\LE$-cylinders and assume $T_i(C)\cap D\neq \emptyset$ for some $D\in\mathcal{D}$.
The sign of $\chi_{D'}\comp T'_i$ is thus positive on all of $C$, i.e.\ $T_i(C)\subset D'\times\rReal$.
Constant sign for the other functions in \eqref{eq: pullback of a cell LE} implies that the sign of the functions $\exp(s_{i,q})-\alpha\comp T'_i$ is constant on $C$ for all $\alpha$ --- indeed, when $\alpha\comp T'_i\leq 0$, it is certainly smaller than $\exp(s_{i,q})$; otherwise, if $\alpha\comp T'_i> 0$, then the sign of $\exp(s_{i,q})-\alpha\comp T'_i$ is the same as that of $s_{i,q}-\log(\alpha\comp T'_i)$.
Thus $T_i(C)\subset D_i$ for some $D_i\in\mathcal{D}$.

\subsubsection{Pulling back the expansions}
\label{sec: LE pulling back expansions}
Let $C$ be the $\LE$-cylinder given in the previous subsection.
Pulling back the expansions \eqref{eq: mu prepared} of the functions $G_{i}$ on the cells $D_i$ by the maps $T_i$, we obtain corresponding expansions for the functions~$g_{i}$ on $C$:
\begin{equation}
	g_{i}(x,y)=\abs{\exp(s_{i,q})-\theta_{i}}^{\lambda_{i}}\cdot a_i(r_i,\exp(s'_i))\cdot u_i(r_i,\exp(s_i)),
\end{equation}
where $\lambda_{i}\in\qRationals$, the function $\theta_{i}\in\LE_{n+1}$ has complexity $(\order[F,n]{1},\polyfd[\format,n][\degree])$, the functions $a_i,u_i$ are in $\FD*[\format][\degree]$, and $u_i(r_i,\exp(s_i))$ is a positive function of logarithmic width at most $\mu/2$ on~$C$.
The function $a_i$ may be taken to be either identically $0$, strictly positive, or strictly negative.
The function $\theta_{i}$ is obtained by composing a function in $\structure$ with $T'_i$.

We distinguish two cases --- the case where $\theta_{i}\not\equiv 0$ and the case where $\theta_{i}\equiv 0$.
In the first case, we have by \eqref{eq: center condition} that
\begin{equation}
	\frac{2}{3}<\frac{\exp(s_{i,q})}{\theta_{i}} < 2.
\end{equation}
We may thus rewrite $\exp(s_{i,q})$ as
\begin{equation}
	\label{eq: restricted exp}
	\exp(s_{i,q}-\log\theta_{i})\cdot \theta_{i},
\end{equation}
where the exponential in \eqref{eq: restricted exp} is restricted to, say, $[-1,1]$.
In particular, we may view this exponential as a function in $\FD[\order{1}][\order{1}]$.

Replacing all occurrences of $\exp(s_{i,q})$ in the parse tree of $f_i$ with \eqref{eq: restricted exp}, we obtain a parse tree for $f_i$ of syntactic complexity lower than $\widetilde{\format}$ (as in \eqref{eq: pullback of a cell LE}), with respect to which $f_i$ is of complexity  $(\order[F,\widetilde\format,n]{1},\polyfd[\format,\widetilde\format,n][\degree])$.
We may now apply \Cref{thm: exp split} inductively (and simultaneously) to those $f_i$ for which this case holds.

It remains to treat the case where $\theta_{i}\equiv 0$.
In this case, we have
\begin{equation}
	\label{eq: zero center}
	g_i(x,y)=\exp(\lambda_{i}s_{i,q})\cdot a_i(r_i,\exp(s'_i))\cdot u_i(r_i,\exp(s_i)).
\end{equation}
For all those $i$ where $f_i=g_i$, we may inductively (and simultaneously) apply \Cref{thm: exp split} to the functions $a_i(r_i,\exp(s'_i))$, using $\mu/2$ as the bound on the logarithmic width, to obtain the desired expression \eqref{eq: LA times exp} for $f_i$.

Otherwise, $g_i$ is the child node of a $\log$-node in the tree of $f_i$.
We first assume that~$a_i>0$.
Using the expression in \eqref{eq: zero center}, we have
\begin{equation}
	\label{eq: log g}
	\log g_i(x,y)=\lambda_{i}s_{i,q}+\log a_i(r_i,\exp(s'_i))+\log u_i(r_i,\exp(s_i)).
\end{equation}
Assuming that $\mu=\order{1}$, we have that $\log u_i\in\FD*[\format][\degree]$.
Replacing all occurrences of $\log g_i$ in the parse tree of $f_i$ with the right-hand side of \eqref{eq: log g}, we obtain a parse tree with respect to which $f_i$ is of complexity  $(\order[F,n]{1},\polyfd[\format,n][\degree])$.
The syntactic complexity of this parse tree is lower than $\widetilde\format$.
If $a_i\leq 0$, we have that $\log g_i(x,y)=0$ by our convention, and so we may similarly replace all occurrences of $\log g_i$ in the parse tree of $f_i$.
We may thus apply \Cref{thm: exp split} inductively (and simultaneously) to all~$f_i$ for which this case holds.
This finishes the proof of \Cref{thm: exp split}. \qed

\subsection{\texorpdfstring{$\LA$-preparation}{LS-preparation}}
\label{sec: LS preparation proof}

In this section, we prove \Cref{prop: LS constant sign}, assuming \Cref{thm: LE cylinder decomposition} holds for all lower values of $n$.
We first show that it is enough to prove the following.
\begin{proposition}
	\label{prop: sign of base}
	Let $f_1,\dots,f_k\in\LA_{n+1}$.
	Then there is a finite decomposition of~$\rReal^{n+1}$ into $\LE$-cylinders, on each of which we have
	\begin{equation}
		\label{eq: function of base}
		\sign f_i (x,y)= \sign a_i(x),
	\end{equation}
	where $a_i\in\LE_n$.

	If $f_1,\dots,f_k$ are of complexity $(F,D)$, we may take the number of cylinders in this decomposition to be $\polyfd[\format,n][\degree,k]$, where each of the cylinders and each of the functions $a_i$ are of complexity $(\order[\format,n]{1},\polyfd[\format,n][\degree])$.
\end{proposition}
Indeed, by induction on $n$, we apply \Cref{thm: LE cylinder decomposition} to $\rReal^n$ and the functions~$a_i$ of \eqref{eq: function of base}, obtaining a decomposition of $\rReal^n$ into $\LE$-cylinders on which the functions $a_i$ have constant sign.
Multiplying each of these cylinders on the right by $\rReal$ we obtain a decomposition of $\rReal^{n+1}$ into $\LE$-cylinders.
We may then take a common refinement (as in \Cref{rem: refine two covers}) of this decomposition with the one obtained by \Cref{prop: sign of base}.

We recall the notion of a \emph{logarithmic scale} from~\cite{vdDriesSpeissegger2002}.
\begin{definition}
	\label{def: log scale}
	Let $C\subset\rReal^{n+1}_{y\ne 0}$ be an $\LE$-cylinder.
	A \emph{log-scale} on $C$ is a finite sequence of functions $y_i\in\LA_{n+1,i}$ satisfying the following conditions.
	\begin{enumerate}
		\item The functions $y_i$ have constant, non-zero sign on $C$;
		\item There exists a sequence of \emph{centers} $\theta_i\in\LE_n$ such that
		\begin{equation}
			\qquad y_0=y-\theta_0(x),\quad y_{i+1}=\log\abs{y_i}-\theta_{i+1}(x);
		\end{equation}
		\item The following \emph{center conditions} are satisfied on $C$.
		If $\theta_0$ or $\theta_{i+1}$ are not identically $0$, then they are nowhere vanishing and
		\begin{equation}
			\label{eq: log scale center condition}
			\qquad 0<\abs{y_0}<\frac{\abs{y}}{2}\quad\text{or}\quad 0<\abs{y_{i+1}}<\frac{1}{2}\abs{\log\abs{y_i}},
		\end{equation}
		respectively.
	\end{enumerate}
\end{definition}

\begin{definition}[$\LA$-prepared]
	\label{def: LA mu prepared}
	Let $C\subset\rReal^{n+1}_{y\ne 0}$ be an $\LE$-cylinder and let~${\mu>0}$.
	A function $f:C\to\rReal$ in $\LA_{n+1,\ell}$ is \emph{$(\LA,\mu)$-prepared} if there exists a log-scale $y_0,\dots,y_\ell$ on $C$ corresponding to centers $\theta_0,\dots,\theta_\ell\in\LE_n$, such that
	\begin{equation}
		\label{eq: LS prepared}
		f(x,y)=\abs{y_0}^{\lambda_{0}}\cdots \abs{y_\ell}^{\lambda_{\ell}}\cdot a(x)\cdot u(\varphi(x),y_\ell,\dots,y_0)
	\end{equation}
	on $C$,	where $\lambda_{0},\dots,\lambda_{\ell}\in\qRationals$, the functions $a$ and the components $(\varphi_{1},\dots,\varphi_{{m}})$ of~$\varphi$ are in $\LE_n$, and $u\in\FD$ is a $\mu$-unit on the image of $C$ under $(\varphi(x),y_\ell,\dots,y_0)$.
	
	When it is clear from context that $f\in\LA_{n+1,\ell}$, we will simply say that it is~\emph{$\mu$-prepared}.
\end{definition}

\Cref{prop: sign of base} follows from the following more explicit statement.
\begin{theorem}[$\LA$-preparation]
	\label{thm: LS preparation}
	Let $f_1,\dots,f_k\in\LA_{n+1,\ell}$ be of complexity $(\format,\degree)$ and let ${\mu>0}$.
	Then there is a decomposition of $\rReal^{n+1}$ into $\polyfd[\format,n,\ell][\degree,k,\mu^{-1}]$-many $\LE$-cylinders, compatible with $\{y=0\}$, such that on each cylinder $C\subset\rReal^{n+1}_{y\neq 0}$ each $f_i$ is $\mu$-prepared.

	We may take each of the cylinders and each of the functions $\theta_{i,j},a_i,\varphi_{i,j}$ in \eqref{eq: LS prepared} to be of complexity $(\order[\format,n,\ell]{1},\polyfd[\format,n,\ell][\degree])$, the $\mu$-units $u_i$ in \eqref{eq: LS prepared} to be in $\FD*[\format][\degree]$, and the height of the exponents $\lambda_{i,j}$ in \eqref{eq: LS prepared} to be at most~$\polyfd$.
\end{theorem}

In order to motivate the notion of a log-scale, as well as for use in the proof of \Cref{thm: LS preparation}, we give \Cref{lem: split y coordinate,lem: kappa,lem: LS complexity drop} below.
For all of these, let $C\subset\rReal^{n+1}$ be an $\LE$-cylinder and let $y_0,\dots,y_\ell$ be a log-scale on $C$, corresponding to centers $\theta_0,\dots,\theta_\ell\in\LE_n$.
Assume that $C$, $y_i$ and $\theta_i$ are all of complexity $(\format,\degree)$.

\begin{lemma}
	\label{lem: split y coordinate}
	Let $\psi,\psi_1,\psi_2\in\LE_n$ (or equal to $\pm\infty$) and let $i=0,\dots,\ell$. 
	Then the sets ${C\cap\{y_i=\psi\}}$ and $C\cap\{\psi_1<y_i<\psi_2\}$ are $\LE$-cylinders.

	If $\psi,\psi_1,\psi_2$ are of complexity $(\format,\degree)$, then the resulting cylinders are of complexity $(\order[\format]{1},\polyfd)$.
\end{lemma}
\begin{proof}
	We treat the case of $C\cap\{\psi_1<y_i<\psi_2\}$, the other case being similar.
	For $i=0$, we have that
	\begin{equation}
		C\cap\{\psi_1<y_0<\psi_2\}=C\cap\{\psi_1+\theta_0<y<\psi_2+\theta_0\}
	\end{equation}
	is the intersection of two $\LE$-cylinders and so is an $\LE$-cylinder (see \Cref{rem: cell vs cylinder}).
	For $i>0$, we have
	\begin{align}
		\begin{split}
			C\cap\{\psi_1<y_i<\psi_2\}&=C\cap\{\psi_1+\theta_i<\log\abs{y_{i-1}}<\psi_2+\theta_i\}\\
			&=C\cap\{\exp(\psi_1+\theta_i)<\abs{y_{i-1}}<\exp(\psi_2+\theta_i)\}.
		\end{split}
	\end{align}
	Since the sign of $y_{i-1}$ is constant on $C$, we may replace $\abs{y_{i-1}}$ by $\pm y_{i-1}$ and proceed by induction on $i$.
	The bounds on the complexity of the resulting cylinders follow directly from the construction.
\end{proof}

\begin{lemma}
	\label{lem: kappa}
	Let $\kappa:C\to\rReal^{n+1}$ be the map given by $\kappa(x,y)=(x,y_1(x,y))$.
	Let $\widehat C\subset\rReal^{n+1}$ be an $\LE$-cylinder of complexity $(\format,\degree)$.
	Then the preimage $\kappa^{-1}(\widehat C)\subset C$ is an $\LE$-cylinder of complexity $(\order[\format]{1},\polyfd)$.
\end{lemma}
\begin{proof}
	Write $\widehat C'$ for the base of $\widehat C$.
	If $\widehat C=(f_1,f_2)_{\widehat C'}$ for $f_1,f_2\in\LE_n$, then
	\begin{equation}
		\kappa^{-1}(\widehat C)=C\cap(\widehat C'\times\rReal)\cap \{f_1<y_1<f_2\}
	\end{equation}
	and so the result follows from the previous lemma.
	The other case, where $\widehat C$ is the graph of an $\LE$-function, is treated in the same way.
\end{proof}
We interpret \Cref{lem: kappa} as meaning ``an $\LE$-cylinder with respect to the coordinates $(x,y_1)$ is also an $\LE$-cylinder with respect to the coordinates $(x,y)$''.
The same result holds with $y_1$ replaced by other terms of the log-scale on $C$, though this case will suffice for the proof of \Cref{thm: LS preparation}.

\begin{lemma}
	\label{lem: LS complexity drop}
	Let $j\in\{0,\dots,\ell-1\}$ and let $\psi\in\LE_n$ be of complexity $(\format,\degree)$.
	Assume that $\abs{y_j/\psi}$ does not vanish and is of logarithmic width at most $\order{1}$ on~$C$.
	Then $y_0,\dots,y_\ell$ may be considered as elements of $\LA_{n+1,\ell-1}$ of complexity $(\order[\format]{1},\polyfd)$.
\end{lemma}
\begin{proof}
	Write
	\begin{equation}
		\label{eq: log depth drop}
		y_{j+1}+\theta_{j+1}=\log\abs{y_j}=\log\abs{\frac{y_j}{\psi}}+\log\abs{\psi}.
	\end{equation}
	Up to an additive constant, the first logarithm on the right-hand side of \eqref{eq: log depth drop} may be taken to be $\restrict{\log}{[1,\order{1}]}\in\FD[\order{1}][\order{1}]$.
	Thus $y_{j+1}\in\LA_{n+1,j}$, which implies $y_{j+2}\in\LA_{n+1,j+1}$ and so on up to $y_\ell\in\LA_{n+1,\ell-1}$.
\end{proof}

The proof of \Cref{thm: LS preparation} will occupy the rest of this section.
We proceed by induction on the logarithmic depth $\ell$ (see \Cref{def: LA function}) of the functions $f_1,\dots,f_k$ in the statement of the theorem.

\subsubsection{\texorpdfstring{The case $\ell=0$}{The case l=0}}
\label{sec: ell eq 0}

Let $f_1,\dots,f_k\in\LA_{n+1,0}$ be of complexity $(\format,\degree)$.
By contracting edges, we may assume that
\begin{equation}
	f_i(x,y)=F_i(\varphi_i (x),y),
\end{equation}
where $F_i\in\FD[\format][\degree]$ and the components of $\varphi_i=(\varphi_{i,1},\dots,\varphi_{i,m_i})\in \LE_n^{m_i}$ are of complexity $(\format,\degree)$.
We may replace without loss of generality $m_i$ by $m=\max_i m_i$.
Consider $\rReal^{m+1}$ with coordinates $(z,w)=(z_1,\dots,z_m,w)$ and the maps $T_i:\rReal^{n+1}\to\rReal^{m+1}$ given by
\begin{equation}
	T_i(x,y)=(\varphi_i(x),y).
\end{equation}
Also write $T'_i$ for $\varphi_i:\rReal^{n}\to\rReal^m$.
The functions $f_i$ factor as $F_i\comp T_i$.

We apply $\structure$-preparation (\Cref{thm: S-preparation}) to $\rReal^{m+1}$ and the functions $F_1,\dots,F_k$, obtaining a decomposition of $\rReal^{m+1}$ into $\polyfd[\format,n][\degree,k,\mu^{-1}]$ cells of suitable complexity, over which each of the functions $F_i$ is $\mu/2$-prepared in the sense of \Cref{def: mu prepared}.
We may also assume that these cells are compatible with $\{w=0\}$.

As in the proof of \Cref{thm: exp split} in \Cref{sec: LE preparation proof}, we proceed in two stages.
First, in \Cref{sec: ell 0 pullback of a cell}, we decompose $\rReal^{n+1}$ using a collection of $\LE$-cylinders, each of them mapped by each of the maps $T_i$ into one of the cells covering $\rReal^{m+1}$.
Then, restricting our attention to one of these cylinders, we will further refine it in \Cref{sec: ell 0 pulling back expansions} by $\LE$-cylinders on which each $f_i$ is of the desired form \eqref{eq: LS prepared}.

\subsubsection{Pullback of a cell}
\label{sec: ell 0 pullback of a cell}

Let $\mathcal{D}$ denote the collection of cells covering $\rReal^{m+1}$, obtained by the application of $\structure$-preparation to the functions $\{F_i\}$.
For each cell $D\in\mathcal{D}$, write $D'$ for its base.
Let $\chi_{D'}$ be the indicator function of $D'$.
The cell $D$ is determined over $D'$ by the sign of up to two functions of the form $w-\alpha(z)$, where $\alpha:D'\to\rReal$ is in $\FD$.
We may assume $\alpha$ is identically $0$ outside of $D'$.
We note that the functions $\chi_{D'}$ and $\alpha$ are in $\FD*[\format][\degree]$.

The pullback of each of the cells $D\in\mathcal{D}$ by each of the maps $T_i$ is thus the $\LE$-cylinder with base $\{\chi_{D'}\comp T'_i>0\}$ and walls determined by the $\LE$-functions $\alpha\comp T'_i$.
We may take a common refinement of these cylinders as in the proof of \Cref{thm: s exp sharp}, applying \Cref{thm: LE cylinder decomposition} by induction on $n$.
(One may also form a common refinement by repeated application of \Cref{rem: refine two covers}, however the number of cylinders obtained in this way is too large).
Let $C$ be one of the resulting cylinders.
For each $i$ we have $T_i(C)\subset D_i$, for some $D_i\in\mathcal{D}$.

We may assume $C$ is compatible with $\{y=0\}$.
If $C$ is a graph over its base (in particular, if $C\subset\{y=0\}$), we are done, since in this case the restrictions of the functions $f_i$ to $C$ depend only on $x$.
We assume from now on that this is not the case.

\subsubsection{Pulling back the expansions}
\label{sec: ell 0 pulling back expansions}
Pulling back the expansions \eqref{eq: mu prepared} of the functions $F_i$ on the cells $D_i$ by the maps $T_i$, we obtain corresponding expansions of the functions $f_i$ on the $\LE$-cylinder $C$:
\begin{equation}
	\label{eq: different centers}
	f_i(x,y)=\abs{y-\theta_{i}(\varphi_i(x))}^{\lambda_{i}}\cdot a_i(\varphi_i(x))\cdot u_i(\varphi_i(x),y),
\end{equation}
where $\theta_i,a_i,u_i\in\FD*[\format][\degree]$ and $u_i(\varphi_i(x),y)$ is positive and of logarithmic width at most $\mu/2$ on $C$.
We also have that each exponent $\lambda_i\in\qRationals$ is of height at most~$\polyfd$.

We will say that a function expanded as in the right-hand side of \eqref{eq: different centers} is \emph{monomialized} with respect to $\abs{y-\theta_i\comp\varphi_i}$.
In this subsection, we explain how to move from these expansions to ones that are monomialized with respect to the same expression for all $i$, and thus obtain the desired expansions \eqref{eq: LS prepared} for the functions~$f_i$.

Denote $\overline{\theta}_i=\theta_i\comp\varphi_i$.
We consider the following ratios, for $i\neq j$.
\begin{equation}
	\label{eq: center ratio}
	\rho_{i,j}=\frac{y-\overline\theta_i}{\overline\theta_i-\overline\theta_j}.	
\end{equation}
We note that $1+\rho_{i,j}=-\rho_{j,i}$.
As we explain more precisely below, whenever $\abs{\rho_{i,j}}$ is large, we have that the centers $\overline\theta_i$ and $\overline\theta_j$ are significantly closer to each other than $y$ is to either one of them.
In this case we may exchange the expressions $\abs{y-\overline\theta_i}$ and $\abs{y-\overline\theta_j}$, up to slightly increasing the logarithmic width of the unit in \eqref{eq: different centers}.
Similarly, whenever $\abs{\rho_{i,j}}$ is of bounded logarithmic width, we may replace $\abs{y-\overline\theta_i}$ by $\abs{\overline\theta_i-\overline\theta_j}$, which does not depend on $y$.

We now explain this in more detail.
Let $c_1,\ldots,c_N$ be an increasing sequence of positive numbers.
We note that
\begin{equation}
	\label{eq: bounded ratio}
	c_k\leq \abs{\frac{y-\overline\theta_i}{\overline\theta_i-\overline\theta_j}}\leq c_{k+1}
\end{equation}
if and only if one of the following holds, depending on the sign of $y-\overline\theta_i$:
\begin{align}
	\label{eq: ratio bounds 1}
	c_k\cdot\abs{\overline\theta_i-\overline\theta_j}+\overline\theta_i&\leq y\leq c_{k+1}\cdot\abs{\overline\theta_i-\overline\theta_j}+\overline\theta_i
	\intertext{or}
	\label{eq: ratio bounds 2}
	-c_{k+1}\cdot\abs{\overline\theta_i-\overline\theta_j}+\overline\theta_i&\leq y\leq -c_{k}\cdot\abs{\overline\theta_i-\overline\theta_j}+\overline\theta_i.
\end{align}

As in the proof of \Cref{thm: s exp sharp}, we apply \Cref{thm: LE cylinder decomposition} to all pairwise differences of the following functions in $\LE_n$: the walls of $C$, the functions~$\overline\theta_i$, and the functions appearing in either the left-hand or right-hand sides of \eqref{eq: ratio bounds 1} and~\eqref{eq: ratio bounds 2}.

Intersecting the base of $C$ with one of the resulting $\LE$-cylinders in $\rReal^n$, we may assume, after omitting duplicates among the functions $\overline\theta_i$, that $\overline\theta_i\neq \overline\theta_j$ over $C$ and so the ratios $\rho_{i,j}$ are well defined for all pairs $(i,j)$.
Furthermore, by splitting $C$ along its fiber, we may assume that, for each pair $(i,j)$, either $\abs{\rho_{i,j}}\leq c_1$, or $\abs{\rho_{i,j}}\geq c_N$, or else there exists some $k$ such that $c_k\leq \abs{\rho_{i,j}}\leq c_{k+1}$ throughout $C$.

Now take the sequence $\{c_i\}$ to be a geometric progression 
\begin{equation}
	q^{-N},\dots,1,\dots,q^N,
\end{equation}
with $\log q=\mu\cdot\polyfd^{-1}$ and $N=\polyfd[\format][\degree,\mu^{-1}]$.
By the discussion above, we may thus assume that for each pair $(i,j)$ at least one of the functions $\abs{\rho_{i,j}}$, $\abs{1+\rho_{i,j}}=\abs{\rho_{j,i}}$, or $\vert{1+\rho_{i,j}^{-1}}\vert$ has logarithmic width at most $\mu\cdot\polyfd^{-1}$ on~$C$.
We proceed to treat each of these cases separately.

If $\abs{\rho_{i,j}}$ has logarithmic width at most $\mu\cdot \polyfd^{-1}$ over $C$, we write
\begin{equation}
	\label{eq: new center mid}
	y-\overline\theta_i=(\overline\theta_i-\overline\theta_j)\cdot\frac{y-\overline\theta_i}{\overline\theta_i-\overline\theta_j}.
\end{equation}
Substituting the right-hand side of \eqref{eq: new center mid} into \eqref{eq: different centers} for $f_i$ yields the expansion
\begin{equation}
	f_i(x,y)=\abs{\overline\theta_i-\overline\theta_j}^{\lambda_{i}}\cdot a_i(\varphi_i(x)) \cdot\abs{\rho_{i,j}}^{\lambda_{i}} u_i(\varphi_i(x),y).
\end{equation}
Since $\abs{\lambda_{i}}=\polyfd$, we have that $\abs{\rho_{i,j}}^{\lambda_{i}} u_i(\varphi_i(x),y)$ is of logarithmic width at most $\mu$.

We perform such a substitution for all indices $i$ such that there exists an index~$j$ for which this case holds.
Restricting attention to the remaining pairs of indices, we may assume from now on that this condition does not hold for any pair $(i,j)$.
In particular, since $\abs{1+\rho_{i,j}}=\abs{\rho_{j,i}}$, it remains to consider pairs $(i,j)$ where $\vert{1+\rho_{i,j}^{-1}}\vert$ is of logarithmic width at most $\mu\cdot \polyfd^{-1}$ over~$C$.

In this case, we write
\begin{equation}
	\label{eq: new center large}
	y-\overline\theta_j=(y-\overline\theta_i)\left(1+\frac{\overline\theta_i-\overline\theta_j}{y-\overline\theta_i}\right).
\end{equation}
Fixing some index $i=i_0$, we substitute the right-hand side of \eqref{eq: new center large} into the expansion \eqref{eq: different centers} of all remaining $f_j$, monomializing them with respect to $\abs{y-\overline\theta_{i_0}}$. 
As before, the resulting units ${\vert{1+\rho_{{i_0},j}^{-1}}\vert}^{\lambda_j}u_j(\varphi_j(x),y)$ are of logarithmic width at most $\mu$.
We thus obtain the desired expansions \eqref{eq: LS prepared} for the functions~$f_i$.

\subsubsection{\texorpdfstring{The case $\ell>0$}{The case l>0}}
Let $f_1,\dots,f_k\in\LA_{n+1,\ell}$ be of complexity $(\format,\degree)$, for $\ell>0$.
We will inductively apply \Cref{thm: LS preparation} to bring the functions $f_i$ to the form \eqref{eq: ell positive starting form} below.
We first illustrate this in a simple example.
\begin{example}
	\label{ex: ell positive example}
	Let $f\in\LA_{n+1,\ell-1}$ and consider $\log f\in\LA_{n+1,\ell}$.
	Applying \Cref{thm: LS preparation} inductively to $f$, we pass to an $\LE$-cylinder $C$ equipped with a log-scale $y_0,\dots,y_{\ell-1}$, determined by centers $\theta_0,\dots,\theta_{\ell-1}$.
	On this cylinder we have
	\begin{equation}
		\label{eq: example inductive expansion of f}
		f(x,y)=\abs{y_0}^{\lambda_0}\cdots\abs{y_{\ell-1}}^{\lambda_{\ell-1}}\cdot a(x)\cdot u(\varphi(x),y_0,\dots,y_{\ell-1}),
	\end{equation}
	where $a\in\LE_n,u\in\structure$ and where $u(\varphi(x),y_0,\dots,y_{\ell-1})$ is positive and of bounded logarithmic width.
	We assume for simplicity that $a>0$ on $C$.

	Taking logarithms on both sides of \eqref{eq: example inductive expansion of f}, we get
	\begin{equation}
		\log f = \sum_{i=0}^{\ell-1}\lambda_{i} \log\abs{y_i}+\log a+\log u (\varphi,y_0,\dots,y_{\ell-1}).
	\end{equation}
	Writing $\log\abs{y_i}=y_{i+1}+\theta_{i+1}$ for $i=0,\dots,\ell-2$ and noting $\log u\in \structure$, we have
	\begin{equation}
		\log f = F(\widetilde\varphi, y_0,\dots, y_{\ell-1},\log\abs{y_{\ell-1}}),
	\end{equation}
	where $F\in\structure$ and $\widetilde\varphi$ is the sequence of $\LE_n$-functions composed of the elements of $\varphi$, of $\log a$, and of the centers $\theta_0,\dots,\theta_{\ell-1}$.
	By \Cref{lem: split y coordinate}, we may pass to an $\LE$-cylinder $C$ on which $\sign(\abs{y_{\ell-1}}-1)$ is constant.
	If $\abs{y_{\ell-1}}\equiv 1$, then $\log f\in\LA_{n+1,\ell-1}$.
	Otherwise, writing $y_\ell=\log\abs{y_{\ell-1}}$, we obtain a $\log$-scale $y_{0},\dots,y_\ell$ on $C$.
\end{example}

We now proceed with the proof.
Returning to our given functions $f_1,\dots,f_k$ and repeating the arguments of \Cref{ex: ell positive example} above for all subtrees of logarithmic depth at most ${\ell-1}$, we obtain a decomposition of $\rReal^{n+1}$ into $\polyfd[\format,n,\ell][\degree,k]$-many $\LE$-cylinders of complexity $(\order[\format,n,\ell]{1},\polyfd[\format,n,\ell][\degree])$.
Let $C$ be one of these cylinders and assume that at least one of the $f_i$ is of logarithmic depth at least $\ell$ on $C$ (otherwise we finish by induction on $\ell$).

As in \Cref{ex: ell positive example}, there is a $\log$-scale $y_0,\dots,y_\ell$ on $C$, determined by centers $\theta_0,\dots,\theta_\ell$, and we have
\begin{equation}
	\label{eq: ell positive starting form}
	f_i=F_i(\varphi_i(x),y_\ell,\dots,y_0),
\end{equation}
where $F_i\in\FD*[\format][\degree]$ and $\varphi_i=(\varphi_{i,1},\dots,\varphi_{i,m})\in\LE_{n}^m$ for $m=\order[\format]{1}$.
The terms of this log-scale, their associated centers, and the components of $\varphi_i$ are all of complexity $(\order[\format,n,\ell]{1},\polyfd[\format,n,\ell][\degree])$.
Note that we order the terms of the $\log$-scale in the entries of $F_i$ such that $y_0$ appears last.

Consider the space $\rReal^{m+\ell+1}$ with coordinates $({z},{w})=(z_1,\dots,z_m,w_\ell,\dots,w_0)$.
Write~${w}'=(w_\ell,\dots,w_{1})$.
Let the maps $T_i:\rReal^{n+1}\to\rReal^{m+\ell+1}$ and $T_i':\rReal^{n+1}\to\rReal^{m+\ell}$ be given by
\begin{equation}
		\ \quad T_i({x},y)=(\varphi_i(x),y_\ell,\dots,y_0),\ \quad T'_i({x},y)=(\varphi_i(x),y_\ell,\dots,y_1).
\end{equation}
The functions $f_i$ factor as $F_i\comp T_i$.
Apply $\structure$-preparation (\Cref{thm: S-preparation}) to~$\rReal^{m+\ell+1}$ and the functions $F_i$.
Let $\mathcal{D}$ be the resulting cellular decomposition of~$\rReal^{m+\ell+1}$.

As in the proof of \Cref{thm: exp split} in \Cref{sec: LE preparation proof} and in the case $\ell=0$ considered in \Cref{sec: ell eq 0}, we proceed in two stages.
In \Cref{sec: ell positive pullback of a cell}, we decompose $\rReal^{n+1}$ using a collection of $\LE$-cylinders, such that each of these cylinders is mapped by each of the maps $T_i$ into one of the cells in $\mathcal{D}$.
Assuming this holds, and restricting our attention to one of these cylinders, we will further refine it in \Cref{sec: ell positive expansion pull back} by $\LE$-cylinders on which each $f_i$ is of the desired form \eqref{eq: LS prepared}.

\subsubsection{Pulling back the expansions}
\label{sec: ell positive expansion pull back}
Continuing with the notation of the previous section, we assume for now that each of the functions $T_i$ maps the cylinder $C$ into one of the cells $D_i\in\mathcal{D}$, which may be different for each $i$.
We reduce to this case in \Cref{sec: ell positive pullback of a cell} below.

On each $D_i$ we have an expansion of the form
\begin{equation}
	\label{eq: ell positive S expansion}
	F_i({z},{w})=\abs{w_0-\theta_i({z},{w}')}^{\lambda_i} \cdot a_i ({z},{w}')\cdot u_i({z},{w}),
\end{equation}
where $\lambda_i\in\qRationals$ are of height at most $\polyfd$, the functions $\theta_i,a_i,u_i$ are in $\FD*[\format][\degree]$, and $u_i$ is a positive function of logarithmic width at most $\mu/2$.
We may assume $\mu=\order{1}$.

We first treat the case where the functions $\theta_1,\dots,\theta_k$ are all identically $0$.
In this case, we pull back the expansions \eqref{eq: ell positive S expansion} by the maps $T_i$ to obtain
\begin{equation}
	\label{eq: zero centers ell positive}
	f_i(x,y) = \abs{y_0}^{\lambda_i}\cdot a_i (\varphi_i(x),y_\ell,\dots,y_1)\cdot u_i(\varphi_i(x),y_\ell,\dots,y_0)
\end{equation}
on $C$.

We consider the functions $a_i (\varphi_i(x),y_\ell,\dots,y_1)$, as well as the functions $y_1,\dots,y_\ell$, as $\LA$-functions with respect to the coordinates $(x,y_1)$.
With respect to these coordinates, these maps are all in $\LA_{n+1,\ell-1}$.
More precisely, these maps all factor as the composition of the map $\kappa$ in \Cref{lem: kappa} and of functions in $\LA_{n+1,\ell-1}$.

We may thus apply \Cref{thm: LS preparation} inductively to these functions to obtain a decomposition of $\rReal^{n+1}$ into $\LE$-cylinders, on each of which these functions are $\mu/2$-prepared (in the sense of \Cref{def: LA mu prepared}) with respect to a suitable log-scale $\overline y_1,\dots,\overline y_\ell$.
By \Cref{lem: kappa}, the resulting cylinders are also $\LE$-cylinders with respect to the original coordinates $(x,y)$.
Replacing $C$ by its intersection with one of these cylinders, it is straightforward to verify that $y_0,\overline y_1,\dots,\overline y_\ell$ is a log-scale.
Finally, substituting the resulting expansions of the functions $a_i (\varphi_i(x),y_\ell,\dots,y_1)$ and $y_1,\dots,y_\ell$ in \eqref{eq: zero centers ell positive}, we obtain the desired expansions~\eqref{eq: LS prepared} for the functions $f_i$.

We now treat the case where at least one of the centers $\theta_i$ is not identically~$0$.
Let $1<M<\polyfd$ be some large integer whose value we will choose later.
Let $m\in\{-M,-M+1,\dots,M-1,M\}$.
Applying \Cref{lem: split y coordinate} to each pair of constant functions $\{m,m+1\}$ and to each $y_j$, we may assume one of the following two cases holds uniformly on $C$ --- either $\abs{y_j}>M>1$ for $j=1,\dots,\ell$, or else, for some such $j$, we have $m\leq y_j \leq m+1$ for some $m$ as above.

If $m \leq y_j \leq m+1$ on $C$, then $\exp(m)\exp(\theta_j) \leq \abs{y_{j-1}} \leq \exp(m+1)\exp(\theta_j)$.
Thus, by \Cref{lem: LS complexity drop}, the functions $y_1,\dots,y_\ell$, and hence also $f_1,\dots,f_k$, are in $\LA_{n+1,\ell-1}$.
We may thus inductively apply \Cref{thm: LS preparation} to $f_1,\dots,f_k$ in this case.

It remains to consider the case where $\abs{y_j}>M>1$ for $j=1,\dots,\ell$.
The assumption that $T_i(C)\subset D_i$ and the center condition \eqref{eq: center condition} imply that
\begin{equation}
	\label{eq: Mboundstart}
	\frac{2}{3}<\frac{y_0}{\theta_i(\varphi_i(x),y_\ell,\dots,y_1)}<2.
\end{equation}
Considering $\theta_i(\varphi_i(x),y_\ell,\dots,y_1)$ with respect to the variables $(x,y_1)$, applying \Cref{thm: LS preparation} inductively, and using \Cref{lem: kappa} as above, we reduce to the case where
\begin{equation}
	\label{eq: y0 and prepared denominator}
	\frac{1}{3}<\frac{y_0}{\abs{y_1}^{\alpha_1}\cdots\abs{y_\ell}^{\alpha_\ell}\cdot A(x)}<3,
\end{equation}
where $\alpha_1,\dots,\alpha_\ell\in\qRationals$ are of height at most $\polyfd$ and $A\in\LE_n$ is of complexity $(\order[\format,n,\ell]{1},\polyfd[\format,n,\ell][\degree])$.
Taking absolute values and logarithms, subtracting $\theta_1$, and rearranging, we have
\begin{equation}
	\label{eq: two bounds for log A minus theta}
	\begin{gathered}
		\log \abs{A}-\theta_1>y_1-(\log 3+\sum_{j=1}^{\ell} \alpha_j \log\abs{y_j}),\\
		\log \abs{A}-\theta_1<y_1-(\log \frac{1}{3}+\sum_{j=1}^{\ell} \alpha_j \log\abs{y_j}).
	\end{gathered}
\end{equation}

Since $M>1$, the center condition \eqref{eq: log scale center condition} implies that $\abs{y_{j}}<\log\abs{y_{j-1}}$ and so $\abs{y_{j}}<\abs{y_{j-1}}$ for $j=2,\dots,\ell$.
Together with the bound $\abs{\alpha_j}<\polyfd$, this implies
\begin{equation}
	\label{eq: iterated log bound}
	\log 3+\sum_{j=1}^{\ell} \abs{\alpha_j} \abs{\log\abs{y_j}}<\polyfd\log\abs{y_1}.
\end{equation}
Thus, for $M>\polyfd$ we have that the right-hand side of \eqref{eq: iterated log bound} is of size at most $\abs{y_1}/\polyfd$.
Substituting into \eqref{eq: two bounds for log A minus theta}, we get that $\log\abs{A}-\theta_1$ and $y_1$ have the same sign on $C$. 
Dividing through by $y_1$ in \eqref{eq: two bounds for log A minus theta}, we get
\begin{equation}
	\label{eq: y1 equivalent to LEn}
	1-\frac{1}{\polyfd}<\abs{\frac{\log \abs{A} -\theta_1}{y_1}}<1+\frac{1}{\polyfd}.
\end{equation}

If $\ell>1$, by \eqref{eq: y1 equivalent to LEn} we may apply \Cref{lem: LS complexity drop} to $y_1$ and $\log\abs{A}-\theta_1$ as before and then finish by applying \Cref{thm: LS preparation} inductively to $f_1,\dots,f_k$.
If $\ell=1$, then the bound $\abs{\alpha_1}<\polyfd$ and \eqref{eq: y0 and prepared denominator}, \eqref{eq: y1 equivalent to LEn} imply that $\abs{y_0/\psi(x)}$ has logarithmic width bounded by $\order{1}$, for a suitable $\psi\in\LE_n$ of complexity $(\order[\format,n,\ell]{1},\polyfd[\format,n,\ell][\degree])$.
Hence we may apply \Cref{lem: LS complexity drop} in this case as well.

\subsubsection{Pullback of a cell}
\label{sec: ell positive pullback of a cell}

Continuing with the notation of the previous sections, it remains to reduce to the case where the cylinder $C$ is mapped by each of the maps $T_i$ into some cell $D_i\in\mathcal{D}$.
Denote by $D'$ the base cell of the cell $D\in\mathcal{D}$.
Let $\chi_{D'}$ be its indicator function.
Considering all functions $\chi_{D'}(\varphi_i(x),y_\ell,\dots,y_1)$ (for all choices of $D$ and of $i$) as elements of $\LA_{n+1,\ell-1}$ with respect to the variables $(x,y_1)$, applying \Cref{thm: LS preparation} inductively, and using \Cref{lem: kappa} as in \Cref{sec: ell positive expansion pull back} above, we may assume that for each $i$ we have $T_i(C)\subset D_i'\times \rReal$, for some $D_i\in\mathcal{D}$.

The cell $D_i$ is determined over $D'_i$ by the sign of at most two functions of the form $w_0-\alpha(z,w'$), where $\alpha:D_i'\to\rReal$ is a continuous function in $\FD*[\format][\degree]$ which may be taken to have constant sign on $D_i'$ (otherwise we split $D'_i$ according to the sign of $\alpha$).
We may assume $\alpha$ vanishes outside $D_i'$.
It remains to reduce to the case where the functions
\begin{equation}
	g_{i,\alpha}=y_0-\alpha(\varphi_i,y_\ell,\dots,y_1)
\end{equation}
have constant sign on $C$.
If $\alpha\equiv 0$, then the sign of $g_{i,\alpha}=y_0$ is already constant on $C$.
We assume from now on that $\alpha>0$ on $D_i'$, the other case being similar.

Since $\alpha$ is continuous on $D_i'$, the restriction of $g_{i,\alpha}$ to any fiber ${C\cap(\{x\}\times \rReal)}$ is also continuous.
In particular, this restriction of $g_{i,\alpha}$ may only change sign if it vanishes at some point along the fiber.
Whenever $g_{i,\alpha}=0$, we have $y_0=\alpha(\varphi_i,y_\ell,\dots,y_1)$ and, in particular, 
\begin{equation}
	\label{eq: alpha interval}
	\frac{1}{2}<\frac{y_0}{\alpha\comp T'_i}<2.
\end{equation}
Let $\widetilde{C}_{i,\alpha}\subset C$ be the set where \eqref{eq: alpha interval} holds.
By continuity, the intersection of $\widetilde{C}_{i,\alpha}$ with each of the fibers ${C\cap(\{x\}\times \rReal)}$ is a relatively open subset of this fiber.
Using \Cref{lem: kappa,lem: LS complexity drop} as in the argument surrounding \eqref{eq: Mboundstart}--\eqref{eq: y1 equivalent to LEn} in \Cref{sec: ell positive expansion pull back}, we reduce to the case where $g_{i,\alpha}=h_{i,\alpha}$ on $\widetilde{C}_{i,\alpha}$, where $h_{i,\alpha}\in\LA_{n+1,\ell-1}$ has complexity $(\order[\format,n,\ell]{1},\polyfd[\format,n,\ell][\degree])$.

Applying \Cref{thm: LS preparation} inductively to the functions $h_{i,\alpha}$, we may assume that their sign is constant on $C$.
Fix $x\in\rReal^n$ and assume that $g_{i,\alpha}$ changes sign on $C\cap(\{x\}\times\rReal)$.
Continuity of the restriction of $g_{i,\alpha}$ to this fiber implies that there exist $(x,t_1), (x,t_2)\in{C}$, with $0<\abs{t_1- t_2}$ arbitrarily small, such that
\begin{equation}
	g_{i,\alpha}(x,t_1)=0\neq g_{i,\alpha}(x,t_2).
\end{equation}
In particular, $(x,t_1)\in \widetilde{C}_{i,\alpha}$.
Choosing $t_2$ close enough to $t_1$, we may assume that also $(x,t_2)\in \widetilde{C}_{i,\alpha}$.
But then we have
\begin{equation}
	h_{i,\alpha}(x,t_1)=0\neq h_{i,\alpha}(x,t_2),
\end{equation}
contradicting the fact that $h_{i,\alpha}$ has constant sign.
This shows that the sign of~$g_{i,\alpha}$ is constant on each fiber $C\cap(\{x\}\times \rReal)$.

We finish by splitting the base $C'$ of $C$ such that $g_{i,\alpha}$ will have constant sign on each resulting $\LE$-cylinder.
For example, if $C=(\psi_1,\psi_2)_{C'}$ for $\psi_1,\psi_2\in\LE_n$, then we apply \Cref{thm: LE cylinder decomposition}, by induction on $n$, to all functions of the form
\begin{equation}
	g_{i,\alpha}\left(x,\frac{\psi_1(x)+\psi_2(x)}{2}\right)\in\LE_n
\end{equation}
and replace $C'$ by its intersection with one of the elements of the resulting decomposition of $\rReal^n$.
This finishes the proof of \Cref{thm: LS preparation}.\qed

\section{Complex cells}
\label{sec: complex cells}

We recall the following from the proofs of the $\LE$-preparation and $\LA$-preparation theorems (\Cref{thm: exp split,thm: LS preparation}) in \Cref{sec: LE preparation proof,sec: LS preparation proof}.
We considered a collection of functions $f_1,\dots,f_k$ defined on some $\LE$-cylinder $C$, along with corresponding maps $T_i$ from $C$ to an auxiliary space $\rReal^N$.
This space was covered by a collection $\{D_j\}$ of o-minimal cells, definable in $\structure$ and obtained by an application of $\structure$-preparation.
We then reduced to the case where each map $T_i$ mapped the cylinder $C$ into one of the cells $D_i$.

The use of different maps $T_i$ for each function $f_i$ is needed in order to ensure the polynomial dependence on the number of functions $k$ in the statements of these theorems, which in turn is required in order to obtain polynomial dependence on the degrees of the given functions.
For the proof of the interpolation theorem (\Cref{thm:C}) and of Wilkie's conjecture (\Cref{thm: wilkie intro}), we will need a variant of this argument which uses a single map $T$ for all functions, and maps their domain into a single cell (see \Cref{thm: single cell}).
The relevant map~$T$ will be a \emph{recursive log--exp scale} as defined in \Cref{sec: recursive scale}.

In addition, we recall that the cellular covers obtained by $\structure$-preparation (as formulated in~\cite[Theorem 5.3]{CarmonAnalyticallyGenerated}) are given by the images of \emph{real complex cells} under \emph{real cellular maps}.
Each of these complex cells $\cell$ is contained in a holomorphic extension $\ext{\cell}{\delta}$, and the corresponding real cellular maps extend holomorphically to $\ext{\cell}{\delta}$.
The extension parameter $0<\delta<1$ controls the size of $\cell$ relative to the hyperbolic metric of $\ext{\cell}{\delta}$ --- smaller values of $\delta$ correspond to a larger extension, that is the cell $\cell$ is hyperbolically smaller inside the corresponding cell $\ext{\cell}{\delta}$.
The function $f\in\structure$ to be prepared is then given by the pullback, along the inverse of the corresponding real cellular map, of a holomorphic function defined on the extended cell $\ext{\cell}{\delta}$ (see \Cref{fig: factor through scale} in \Cref{sec: prep thms and complex cells intro}).
This holomorphic function may be taken to either vanish identically or else vanish nowhere on $\ext{\cell}{\delta}$.
We refer the reader to~\cite[Section 2.2 and Appendix A]{BinyaminiCarmonNovikovComplexCells} for a review of the definitions and notation that we need regarding complex cells, as well as to~\cite{BinyaminiNovikov2019} for more on complex cells.

In \Cref{sec: point counting}, we relate the algebraic points of bounded height and degree on an $\LE$-cell $X$ to points on a real complex cell $\cell$ on which several holomorphic functions attain algebraic values of bounded height and degree.
In order to construct a polynomial whose zero-set interpolates these points, it will be enough to construct a polynomial which is uniformly sufficiently small on $X$.

The classical version of the arguments we will use in \Cref{sec: algebraic point interpolation} to construct this polynomial depends delicately on the dimension of the set on which one counts algebraic points.
However, the complex cell $\cell$ might have dimension much larger than that of $X$.
In this section, we explain how to reduce to the case where at most $\dim X$ of the coordinates of $\cell$ are discs, while the rest are points, annuli, punctured discs, or disc complements.
One of the crucial features of the hyperbolic geometry of complex cells inside their extensions is that we may take the extension parameters in non-disc coordinates to be exponentially smaller than those in the disc coordinates.
This essentially means that non-disc coordinates, while contributing to the dimension of the complex cell, may be taken to be ``too  thin'' to affect the estimates we need.

We have the following lemma in this direction.
\begin{lemma}
	\label{lem: thin annuli}
	Let $0<\delta<\order{1}$ and let $t>\order*{1}$.
	Let $\boldsymbol{\delta}$ denote an extension parameter which is of size $\delta$ for disc coordinates and of size $\delta^t$ for all other coordinates.
	
	Let $\ext{\cell}{1/2}\subset\cComplex^\ell$ be a real complex cell in $\FD[\format][\degree]$.
	Then there is a real cellular cover $\{f_j:\ext{\cell_j}{\boldsymbol{\delta}}\to\ext{\cell}{1/2}\}$ of size $\poly_\ell(t,\delta^{-1})$, where each $f_j$ is in $\FD*[\format][\degree]$.
\end{lemma}
\begin{proof}
	The proof is the same as that of~\cite[Lemma 94]{BinyaminiNovikov2019}.
\end{proof}
The analogous statement for complex cells in $\Ran$ is~\cite[Lemma 94]{BinyaminiNovikov2019}.

\begin{remark}
	It is more natural to formulate \Cref{lem: thin annuli} in terms of \emph{hyperbolic $\hyperbolicParameter{\rho}$-extension parameters} rather than the euclidean extension parameters $\delta$, as in~\cite[Theorem 9]{BinyaminiNovikov2019}.
	However, for the sake of brevity and clarity, we will continue with the euclidean normalization, as in~\cite[Appendix B.1]{BinyaminiNovikov2019}.
\end{remark}

\subsection{Recursive log--exp scales}
\label{sec: recursive scale}
In this section, we extend the definitions of syntactic complexity and of logarithmic scales (\Cref{def: log scale,def: syntactic complexity}).
This will make it more convenient to work with all coordinates at once, rather than just with the last coordinate as in \Cref{sec: prepartion theorems}.

\begin{definition}
	\label{def: log exp scale}
	Let $C\subset\rReal^{n+1}$ (with coordinates $x_1,\dots,x_n,y$) be an $\LE$-cylinder compatible with $\{y=0\}$.
	If $C$ is not the graph of a function over its base, then a \emph{log--exp scale on~$C$} is a finite sequence of functions 
	\begin{equation}
		S_y=(L_\ell,\dots, L_0,\exp(E_1),\dots,\exp(E_m)),	
	\end{equation}
	where the \emph{logarithmic part} $(L_0,\dots, L_\ell)$ is a log--scale on $C$ (with respect to~$y$) and the \emph{exponential part} $(E_1,\dots,E_m)$ consists of $\LE$-functions of weakly increasing exponential level (with respect to $y$), i.e.\ $e(E_{i+1})\geq e(E_i)$ in the notation of \Cref{def: syntactic complexity}.
	The \emph{length} $\ell(S_y)$ of $S_y$ is $\ell+1+m$.
	Otherwise, if $C$ is a graph, then a log--exp scale on $C$ is an empty sequence of functions and has length~$0$.
\end{definition}
We note that we order the terms of a log--exp scale from ``most logarithmic'' to ``most exponential'', unlike in \Cref{def: log scale}.

\begin{definition}
	\label{def: recursive scale}
	Let $C\subset\rReal^n$ be an admissible $\LE$-cell (see \Cref{def:LE cell}).
	Denote its base by~$C'$.
	A \emph{recursive log--exp scale on $C$} is a tuple $S=(S_{x_1},\dots,S_{x_n})$, where $S'=(S_{x_1},\dots,S_{x_{n-1}})$ is a recursive log--exp scale on $C'$ and $S_{x_n}$ is a log--exp scale on $C$.
	We set the \emph{length} $\ell(S)$ of $S$ to be $\ell(S_{x_1})+\cdots+\ell(S_{x_n})$.
\end{definition}

In what follows, we will abuse notation and identify between a recursive log--exp scale $S=(S_{x_1},\dots, S_{x_n})$ on an $\LE$-cell $C$ and the function $S:C\to\rReal^{\ell(S)}$ obtained by concatenating the components of $S$.
For a tuple of functions $f=(f_1,\dots,f_m)$ and $\lambda\in\rReal^m$, we write $\abs{f}^\lambda$ for the product~$\prod_{i=1}^m \abs{f_i}^{\lambda_i}$.
\begin{definition}
	Let $C\subset\rReal^{n}$ be an admissible $\LE$-cell, let $f_1,\dots,f_k:C\to\rReal$ be $\LE$-functions, and let  $S=(S_{x_1},\dots,S_{x_n})$ be a recursive log--exp scale on~$C$.
	We say that $f_1,\dots,f_k$ are \emph{prepared} with respect to $S$ if each~$f_i$ is either identically $0$, or else of the form
	\begin{equation}
		\label{eq: recursively prepared}
		f_i=\pm \abs{S}^{\lambda_{i}} \cdot u_i(S),
	\end{equation}
	where $\lambda_{i}\in \qRationals^{\ell(S)}$ and the function $u_i\in\structure$ is a positive function bounded away from $0$ and infinity.
\end{definition}

Unlike in \Cref{sec: prepartion theorems} (and somewhat more in line with the treatment in \cite{vdDriesSpeissegger2002}), the preparation theorem of \Cref{sec: single cell} factors a collection of $\LE$-functions through a single $\structure$-definable o-minimal cell in an auxiliary space.
We will need the following definitions.
\begin{definition}
	Let $f_1,\dots,f_k\in\LE_n$ with respect to some given parse trees.
	Consider the parse tree with an $\structure$-node at the root, whose child nodes are the given parse trees of $f_1,\dots,f_k$.
	We may represent each $f_i$ with this tree by using the $i$-th coordinate function on $\rReal^k$ as the root.
	We will refer to this tree as the \emph{common tree} of $f_1,\dots,f_k$.
\end{definition}
We remark that a common tree as above has tree-format at least~$k$, and so we do not expect polynomial dependence on $k$ whenever using this representation of the functions $f_1,\dots,f_k$.
We will later only use common trees for the coordinate functions on some $\LE$-cell in $\rReal^n$, and so one may think of $k$ as being at most $n$ in what follows.
We also note that each $f_i$ in this representation has the same collection of maximal subtrees.

Finally, for the next two definitions, let $S=(S_{x_1},\dots,S_{x_n})$ be a recursive log--exp scale on an admissible $\LE$-cell $C\subset\rReal^n$.
\begin{definition}
	The \emph{recursive syntactic complexity} of $S$ is the tuple obtained by concatenating the following:
	\begin{enumerate}
		\item The syntactic complexity (with respect to $x_n$) of the common tree of the terms of $S_{x_n}$,
		\item The length of $S_{x_n}$, and
		\item The recursive syntactic complexity of $S'=(S_{x_1},\dots,S_{x_{n-1}})$
	\end{enumerate}
\end{definition}

\begin{definition}
	A \emph{maximal index} for $S$ is one of the numbers
	\begin{equation}
		1,\ell(S_{x_1})+1,\dots,\ell(S_{x_1})+\cdots+\ell(S_{x_{n-1}})+1.	
	\end{equation}
	These indices correspond to ``maximally logarithmic'' terms in $S$, with respect to each variable.
	In particular, the scale $S$ may have at most $\dim C$ maximal indices.
\end{definition}

The main result of this subsection is the following.
\begin{theorem}
	\label{thm: recursive preparation}
	Let $f_1,\dots,f_k\in\LE_{n}$.
	Then there is a finite decomposition of $\rReal^{n}$ into admissible $\LE$-cells, and for each $\LE$-cell $C$ there is a recursive log--exp scale $S:C\to\rReal^{N}$ such that $f_1,\dots,f_k$ are prepared with respect to $S$.

	If $f_1,\dots,f_k$ are of complexity $(\format,\degree)$, we may take the number of cells in this decomposition to be $\polyfd[\format,n,k][D]$, the $LE$-cells and the terms of the corresponding scales to be of complexity $(\order[\format,n,k]{1},\polyfd[\format,n,k][\degree])$, the length of each scale to be at most $\order[\format,n,k]{1}$, the exponents $\lambda_i$ in \eqref{eq: recursively prepared} to be of height at most $\polyfd[\format,n,k][\degree]$, and the unit $u_i$ in \eqref{eq: recursively prepared} to be in $\FD*[\format,n,k][\degree]$.
\end{theorem}
\begin{proof}
	Throughout the proof, we consider the functions $f_1,\dots,f_k$ restricted to some $\LE$-cell $C$ (taking it to be all of $\rReal^{n}$ at the outset).
	On cells where some coordinate, say $x_j$, is a function of the previous coordinates, we substitute this function in place of $x_j$ in $f_1,\dots,f_k$ and proceed by induction on $n$.
	We may thus restrict our attention from now on to cells of dimension $n$, that is we may assume no coordinate of these cells is of graph type.

	If the functions $f_i$ are all in $\LA_{n}$, we may apply $\LA$-preparation (\Cref{thm: LS preparation}) to reduce to the case where they are given in the form
	\begin{equation}
		\label{eq: log scale partial preparation}
		f_i=\abs{S_{x_n}}^{\lambda_i} \cdot \varphi_{i,0} \cdot u_i(\varphi_{i,1},\dots,\varphi_{i,m},S_{x_n}),
	\end{equation}
	where $S_{x_n}$ is a log--scale (and so, up to reordering, a log--exp scale) of length $N=\order[\format]{1}$ on $C$ and $\lambda_i\in\qRationals^N$, the function $u_i\in\structure$ is positive and bounded away from $0$ and infinity, and $\varphi_{i,0},\dots,\varphi_{i,m}\in\LE_{n-1}$, for $m=\order[\format]{1}$.
	By induction on~$n$, we may reduce to the case where $\varphi_{1,0},\dots,\varphi_{k,m}$ are prepared with respect to a recursive log--exp scale $(S_{x_1},\dots,S_{x_{n-1}})$, and thus the functions~$f_i$ are prepared with respect to the recursive log--exp scale $(S_{x_1},\dots,S_{x_{n}})$.
	The relevant effective bounds follow directly from the construction.

	Otherwise, if one of the functions $f_i$ is not in $\LA_{n}$, we proceed by induction on syntactic complexity (with respect to $x_n$) of their common tree, as follows.
	We repeat the arguments of the proof of LE-preparation in \Cref{sec: LE preparation proof}, constructing maps $T_i\in \LE_n$ which map the $\LE$-cell $C$ into a cell $D_i\in\structure$.
	Using a common tree for $f_1,\dots,f_k$, the maps $T_i$ are the same for all $i$.
	Thus, after possibly passing to a smaller $\LE$-cell, we obtain expansions
	\begin{equation}
		\label{eq: log exp scale partial preparation}
		f_i=\exp(s_q)^{\lambda_i} \cdot \varphi_{i,0}\cdot u_i(\varphi_{i,1},\dots,\varphi_{i,m},\exp(s_q)),
	\end{equation}
	where $s_q,\varphi_{i,0},\dots,\varphi_{i,m}\in\LE_{n}$ and the maximal subtrees of $\varphi_{1,0},\dots,\varphi_{k,m}$ are subtrees of the common tree for $f_1,\dots,f_k$.
	The common tree of $\varphi_{1,0},\dots,\varphi_{k,m}$ has lower syntactic complexity (with respect to $x_n$) than that of the common tree of $f_1,\dots,f_k$ --- indeed, it either has lower exponential level or else fewer maximal subtrees (as it is missing $\exp(s_q)$).
	By induction, we may refine $C$ such that the functions $\varphi_{1,0},\dots,\varphi_{k,m}$ are prepared with respect to a recursive log--exp scale $S'$.
	Concatenating $S'$ with $\exp(s_q)$ yields a recursive log--exp scale $S$ on $C$ with respect to which the functions $f_1,\dots,f_k$ are prepared.
	The relevant effective bounds follow directly from the construction.
\end{proof}

By using different scales for different functions, it is possible to recover the polynomial dependence on the number of functions, as follows.
\begin{corollary}
	\label{cor: different scales poly k 1}
	Let $f_1,\dots,f_k\in\LE_{n}$.
	Then there is a finite decomposition of $\rReal^{n}$ into admissible $\LE$-cells, and for each $\LE$-cell $C$ there are recursive log--exp scales $S_i:C\to\rReal^{N}$ such that each $f_i$ is prepared with respect to $S_i$, respectively.

	If $f_1,\dots,f_k$ are of complexity $(\format,\degree)$, we may take the number of cells in this decomposition to be $\polyfd[\format][D,k]$, the $LE$-cells and the terms of the corresponding scales to be of complexity $(\order[\format]{1},\polyfd[\format][\degree])$, the length of each scale to be at most $\order[\format]{1}$, the exponents $\lambda_i$ in \eqref{eq: recursively prepared} to be of height at most $\polyfd[\format][\degree]$, and the unit $u_i$ in \eqref{eq: recursively prepared} to be in $\FD*[\format][\degree]$.
\end{corollary}
\begin{proof}
	We apply \Cref{thm: recursive preparation} to each of the functions $f_i$ separately.
	We then apply \Cref{thm: sharp cd} to the resulting $k\cdot \polyfd[\format][\degree]$-many $\LE$-cells to obtain the desired decomposition of $\rReal^n$.
\end{proof}

\subsection{Factoring through a single complex cell}
\label{sec: single cell}

We now state the complete preparation theorem we need for our diophantine applications in \Cref{sec: point counting}.
As in~\cite[Definition 2.19]{BinyaminiCarmonNovikovComplexCells}, we will say that $\phi:\cComplex^N\to\cComplex^N$ is a \emph{power map} if each coordinate $\phi_j$ is of the form $\phi_j(z_1,\dots,z_N)=\pm z_j^{q_j}$ for some non-zero $q_j\in\zIntegers$.

\begin{definition}[Prepared via complex cells]
	\label{def: complex cell prepared}
	Let $C\subset\rReal^n$ be an admissible $\LE$-cell.
	Let $0<\delta<1$ and let $t>1$.
	Let $\boldsymbol{\delta}$ be as in \Cref{lem: thin annuli} with respect to $t$ and $\delta$.
	An $\LE$-function $f:C\to\rReal$ is \emph{prepared via complex cells} on $C$ if there exist a recursive log--exp scale $S$ of length $N$ on $C$, a definable real complex cell $\ext{\cell}{\boldsymbol{\delta}}\subset\cComplex^{N}$, a power map $\phi:\ext{\cell}{\boldsymbol{\delta}}\to\cComplex^{N}$, and a definable real holomorphic function $g:\ext{\cell}{\boldsymbol{\delta}}\to\cComplex$, compatible with $0$, such that the following conditions hold.
	\begin{enumerate}
		\item We have that $S(C)\subset\phi(\rRealPos\cell)$ and ${f=g\comp\phi^{-1}\comp S}$ on $C$.
		\item The fibers of $\cell$ may only be discs at maximal indices of $\structure$.
	\end{enumerate}
	In particular, the number of disc fibers of $\cell$ is at most the dimension of $C$.
\end{definition}

\begin{remark}
	Recall from~\cite[Lemma 5.1]{CarmonAnalyticallyGenerated} that a nowhere-vanishing holomorphic function $g$ on a complex cell $\ext{\cell}{\delta}$ is equal to the product of a power map with effectively bounded exponents and a unit with effectively bounded logarithmic width (when restricted to $\cell$).
\end{remark}

\begin{theorem}[Preparation via complex cells]
	\label{thm: single cell}
	Let $f_1,\dots,f_k\in\LE_{n}$ be of complexity $(\format,\degree)$.
	Let $0<\delta<\order{1}$ and let $t>\order*{1}$.
	Then there is a decomposition of $\rReal^n$ into $\polyfd[\format,n,k][D,t,\delta^{-1}]$ admissible $\LE$-cells, on each of which each function $f_i$ is prepared via a common complex cell $\ext{\cell}{\boldsymbol{\delta}}$ and a common $\log$--$\exp$-scale $S$.

	We may take the $LE$-cells and the terms of the corresponding scales to be of complexity $(\order[\format,n,k]{1},\polyfd[\format,n,k][\degree])$, the length of each scale to be at most $\order[\format,n,k]{1}$, the complex cells and the relevant holomorphic functions $g_i$ to be in $\FD*[\format,n,k][\degree]$, and the exponents in the relevant power maps $\phi$ to be of size at most $\polyfd[\format,n,k][\degree]$.
\end{theorem}

\begin{remark}
	In \Cref{thm: single cell} above, the complex cell $\ext{\cell}{\boldsymbol{\delta}}$ and the scale $S$ depend on the $\LE$-cell but are the same for all functions $f_i$ (restricted to this cell).
	This is needed for the arguments of \Cref{sec: point counting}, but comes at the cost of losing polynomial dependence on the number of functions, $k$.
	For a similar statement with polynomial dependence on $k$ but different complex cells and scales for each function, see \Cref{cor: different scales poly k 2}.
\end{remark}

\begin{proof}[Proof of \Cref{thm: single cell}]
	As in the proof of \Cref{thm: recursive preparation}, we consider the functions $f_1,\dots,f_k$ restricted to some $\LE$-cell $C$ (taking it to be all of $\rReal^{n}$ at the outset).
	On cells where some coordinate, say $x_j$, is a function of the previous coordinates, we substitute this function in place of $x_j$ in $f_1,\dots,f_k$ and proceed by induction on $n$.
	We may thus restrict our attention from now on to cells of dimension $n$.

	By \Cref{thm: recursive preparation}, we may assume that there exists a recursive log--exp scale $S$ on $C$ such that $f_1,\dots,f_k$ are prepared with respect to $S$.
	We proceed by induction on the recursive syntactic complexity of $S$.

	Let $\ell=\order[\format,n,k]{1}$ be the length of $S$.
	Each of the functions $f_i$ factors as $F_i\comp S$ for some $F_i:\rReal^\ell\to\rReal$ in $\FD*[\format,n,k][\degree]$ (by the bounds in the statement of \Cref{thm: recursive preparation}).
	Applying $\structure$-preparation to $\rReal^\ell$ and to the functions $F_i$ and to the coordinate functions $x_1,\dots,x_\ell$, we obtain a cover of $\rReal^\ell$ by the images of the positive real parts of real complex cells $\cell$ under real cellular maps $\phi$.
	We may assume by \Cref{lem: thin annuli} that these complex cells admit $\boldsymbol{\delta}$-extensions.
	By~\cite[Theorem 7, CPrT]{BinyaminiNovikov2019}, each map~$\phi$ may be taken to be such that its coordinates~$\phi_j$ are of the form
	\begin{equation}
		\phi_j(z_1,\dots,z_N)=\pm z_j^{q_j}+\varphi_j(z_1,\dots,z_{j-1}),
	\end{equation}
	where $q_j\in\zIntegers\setminus\{0\}$ and $\varphi_j$ is holomorphic and compatible with $0$.
	The restrictions of the functions $\varphi_j$ to the positive real part of $\cell$ give the center functions of \Cref{thm: S-preparation}.	

	Applying \Cref{thm: sharp cd} to all preimages $S^{-1}(\phi(\cell))$ and to the $\LE$-cells covering $\rReal^n$, we reduce to the case where $S(C)\subset \phi(\cell)$ for one of these complex cells $\cell$.
	Let $S_j$ denote the $j$-th coordinate of $S:C\to\rReal^\ell$.
	Thus, for each~$\varX=(x_1,\dots,x_n)\in C$, there exists a unique $\varZ=(z_1,\dots,z_\ell)\in\rRealPos\cell$ such that
	\begin{equation}
		S_j(\initial{\varX}{1}{n})=\phi_j(\initial{\varZ}{1}{\ell})=\pm z_j^{q_j}+\varphi_j(\initial{\varZ}{1}{j-1})
	\end{equation}
	for all $j$.
	Rearranging, we get
	\begin{equation}
		\label{eq: z in terms of x}
		z_j=(\pm(S_j(\initial{\varX}{1}{n})-\varphi_j(\initial{\varZ}{1}{j-1})))^\frac{1}{q_j}.
	\end{equation}
	Recursively substituting \eqref{eq: z in terms of x} into itself, we may consider each coordinate $z_j$ as an $\LE$-function of $\varX$.

	We now consider the case where $\varphi_j$ is not identically $0$ for some non-maximal index $j$ of $S$.
	In this case, we will replace $S$ by a recursive log--exp scale of lower recursive syntactic complexity, as follows.
	If the $j$-th coordinate of $S$ is of the form $\exp(s)$ for some $s\in\LE_{m}$, we return to the expansion corresponding to \eqref{eq: log exp scale partial preparation} at the relevant stage of the construction of $S$, from which $\exp(s)$ was added to the scale $S$.
	As in \Cref{sec: LE pulling back expansions}, we may rewrite the parse tree of the function on the left-hand side of \eqref{eq: log exp scale partial preparation} --- replacing $\exp(s)$ by $\exp(s-\log\sigma)\cdot\sigma$, for a suitable $\sigma\in\LE_m$ and using a restricted exponential.
	We proceed to construct a new recursive log--exp scale preparing $f_1,\dots,f_k$ from this point, and the resulting scale will be of lower recursive syntactic complexity than $S$ since it is missing $\exp(s)$ and no additional maximal subtrees are created in this way.

	In the same way, if $\varphi_j$ is not identically zero for a non-maximal logarithmic term in $S$, we return to the expansion corresponding to \eqref{eq: log scale partial preparation} at the relevant stage of the construction of $S$.
	After possibly passing to a smaller $\LE$-cell, we replace the log-scale on the right-hand side of \eqref{eq: log scale partial preparation} by a shorter one, using \Cref{lem: LS complexity drop} as in \Cref{{sec: ell positive expansion pull back}}.
	We proceed to construct a new recursive log--exp scale preparing $f_1,\dots,f_k$ from this point, and the resulting scale will be of lower recursive syntactic complexity than $S$ since the length of the corresponding log--exp scale is reduced.

	Proceeding in this way, by induction on the recursive syntactic complexity of~$S$, we may assume that $\varphi_j$ is identically $0$ for all $j$ which are not maximal indices of $S$.
	Since the image $\phi(\cell)$ is, by construction, compatible with the coordinate hyperplanes of $\cComplex^\ell$, the corresponding coordinates of $\cell$ are not of type~$\disc$.

	Finally, if $\varphi_j$ is not identically zero for $j$ a maximal index of $S$, we replace the component $S_j$ of $S$ with its translation $S_j(\varX)-\varphi_j(\initial{\varZ}{1}{j-1})$, where we consider $\varZ$ as an $\LE$-function of $\varX$ as in \eqref{eq: z in terms of x}.
	In this way, $\varphi_j$ may be viewed as a function of $S_1(\varX),\dots,S_{j-1}(\varX)$.
	Since $j$ is a maximal index, if $S_j$ is a function of  the coordinates $\initial{\varX}{1}{n_j}$, then the terms $S_1,\dots,S_{j-1}$ depend only on the previous coordinates $\initial{\varX}{1}{n_{j}-1}$.
	The center conditions \eqref{eq: center condition}, \eqref{eq: log scale center condition} imply that
	\begin{equation}
		0<\abs{S_j - \varphi_j}<\frac{1}{2}\abs{S_j}<\frac{1}{2}\abs{\log\abs{S_{j+1}}}.
	\end{equation}
	Thus replacing $S_j$ by $S_j(\varX)-\varphi_j(\initial{\varZ}{1}{j-1})$ still yields a log--scale with respect to the variable $x_{n_j}$.
	In this way, we may assume that the map $\phi$ is a power map.
	\qedhere
\end{proof}

As in \Cref{cor: different scales poly k 1}, by using different scales and complex cells for different functions, it is possible to recover the polynomial dependence on the number of functions.

\begin{corollary}
	\label{cor: different scales poly k 2}
	Let $f_1,\dots,f_k\in\LE_{n}$ be of complexity $(\format,\degree)$.
	Let ${0<\delta<\order{1}}$ and let ${t>\order*{1}}$.
	Then there is a decomposition of $\rReal^n$ into $\polyfd[\format][D,k,t,\delta^{-1}]$ admissible $\LE$-cells, on each of which each function $f_i$ is prepared via a complex cell $\ext{\cell_i}{\boldsymbol{\delta}}$ and a $\log$--$\exp$-scale $S_i$.

	We may take the $LE$-cells and the terms of the corresponding scales to be of complexity $(\order[\format]{1},\polyfd[\format][\degree])$, the length of each scale to be at most $\order[\format]{1}$, the complex cells and the relevant holomorphic functions $g_i$ to be in $\FD*[\format][\degree]$, and the exponents in the relevant power maps $\phi$ to be of size at most $\polyfd[\format][\degree]$.
\end{corollary}

\begin{proof}
	We apply \Cref{thm: single cell} to each of the functions $f_i$ separately.
	We then apply \Cref{thm: sharp cd} to the resulting $k\cdot \polyfd[\format][\degree,t,\delta^{-1}]$-many $\LE$-cells to obtain the desired decomposition of $\rReal^n$.	
\end{proof}

\subsubsection{The case of \texorpdfstring{$\Ranexp$}{Ran,exp}}
\label{sec: Ranexp preparation}
In the non-sharp setting of $\Ranexp$, we consider $\LE$-functions and $\LE$-cells with respect to $\structure=\Ran$.
We have the following version of \Cref{thm: single cell}.
\begin{theorem}
	\label{thm: single cell Ranexp}
	Let $f_1,\dots,f_k:\rReal^n\to\rReal$ be definable in $\Ranexp$.
	Let $0<\delta<\order{1}$ and let $t>\order*{1}$.
	Then there is a finite decomposition of $\rReal^n$ into  admissible $\LE$-cells, on each of which each function $f_i$ is prepared via a common complex cell $\ext{\cell}{\boldsymbol{\delta}}$ and a common $\log$--$\exp$-scale $S$.
	
	Fixing $k$, if the functions $f_i\in\{f_\lambda\}_{\lambda\in\Lambda}$ belong to a single definable family $\Lambda$, then we may take the number of cells in this decomposition to be $\polyfd[\Lambda,k][t,\delta^{-1}]$, the $LE$-cells and the terms of the corresponding scales to belong to a single definable family, the length of each scale to be uniformly bounded over $\Lambda$, the complex cells and the relevant holomorphic functions $g_i$ to belong to a single definable family, and the exponents in the relevant power maps $\phi$ to be uniformly bounded over~$\Lambda$.
\end{theorem}
\begin{proof}
	The proof is similar to that of \Cref{thm: single cell}, ignoring questions of effectivity.
	Uniformity in definable families is obtained as follows (cf.\ \cite[Remark 28]{BinyaminiNovikov2019}).
	First, we apply the theorem to the total space $X$ of the definable family~$\Lambda$, obtaining a decomposition into admissible $\LE$-cells.
	Let $C$ be one of these cells and let $S$ be the corresponding recursive $\log$--$\exp$-scale on $C$.
	
	Assume for simplicity that the parameters of the family $\Lambda$ are given by the first coordinate $x_1$ of $X$.
	Fixing the value $x_1=\lambda$ in order to specialize to a member $X_{\lambda} $ of the family, we also specialize $C$ to an admissible $\LE$-cell $C_\lambda$ and the functions $f_1,\dots,f_k$ to functions $f_{1,\lambda},\dots,f_{k,\lambda}$.

	By fixing the first $\order[\Lambda]{1}$ coordinates of $S$ which correspond to the variable~$x_1$ (according to the choice $x_1=\lambda$), we obtain a recursive $\log$--$\exp$-scale $S_\lambda$ on $C_\lambda$.
	The cellular structure of the complex cell $\ext{\cell}{\boldsymbol{\delta}}$ corresponding to $C$ allows us to also fix its first $\order[\Lambda]{1}$ coordinates and specialize to a complex cell $\ext{\cell_\lambda}{\boldsymbol{\delta}}$ such that $S_\lambda(C_\lambda)\subset \phi_\lambda(\rRealPos\cell_\lambda)$, where $\phi_\lambda$ is the restriction of the relevant power map~$\phi$ to $\ext{\cell_\lambda}{\boldsymbol{\delta}}$.
	Thus each function $f_{i,\lambda}$ is given by $g_{i,\lambda}\comp \phi_{\lambda}^{-1}\comp S_\lambda$, where $g_{i,\lambda}$ is the restriction of the relevant holomorphic map $g_i:\ext{\cell}{\boldsymbol{\delta}}\to\cComplex$ to $\ext{\cell_\lambda}{\boldsymbol{\delta}}$.
\end{proof}

\section{Interpolation and point counting}
\label{sec: point counting}
In this section, we prove Wilkie's conjecture for $\structure(\exp)$ (\Cref{thm: wilkie intro}).
We prove \Cref{thm:C} and its analogue for $\structure(\exp)$-definable sets in \Cref{sec: algebraic point interpolation} below.
We repeat the statement of \Cref{thm: wilkie intro} here for the convenience of the reader.
For a set $X\subset\rReal^n$, we write $X(g,H)$ for the set of algebraic points in $X$ of degree at most $g$ (that is, points $x\in X$ such that $[\qRationals(x_1,\dots,x_n):\qRationals]\leq g$) and multiplicative Weil height at most $H$.
We write $X^{\operatorname{alg}}$ for the union of all connected, positive dimensional, semialgebraic subsets of $X$, and set $X^{\operatorname{trans}}=X\setminus X^{\operatorname{alg}}$.
Our second main result is the following, where we use the FD-filtration for $\structure(\exp)$ constructed in \Cref{sec: so min}.
\begin{theorem}
	\label{thm: wilkie}
	Let $X\subset \rReal^n$ be  in $\structure(\exp)_{\format,\degree}$.
	Then 
	\begin{equation}
		\label{eq: Wilkie conjecture bound}
		\# X^{\operatorname{trans}}(g,H)\leq\poly_{\format}(\degree,g,\log H).
	\end{equation}
\end{theorem}
\begin{remark}
	Let $\structure$ be the structure analytically generated from $\RrPfaff$ (see~\cite[Definition 2.1]{CarmonAnalyticallyGenerated} for this notion).
	This structure contains the restricted exponential function since both it and the restricted sine function are Pfaffian.
	\Cref{thm: wilkie} for $\structure(\exp)$ then implies, in particular, Wilkie's original conjecture for $\Rexp$.
	This conjecture was recently proved by Binyamini--Novikov--Zak in~\cite{BinyaminiNovikovZak2024} with bounds of the form $\poly_X(g,\log H)$.
\end{remark}

The proof of \Cref{thm: wilkie} relies on the following analogue of \Cref{thm:C} for $\structure(\exp)$-definable sets.
\begin{theorem}
	\label{thm: interpolation}
	Let $X\subset [0,1]^n$ in $\structure(\exp)_{\format,\degree}$.
	Assume that $\dim X<n$.

	Then $X(g,H$) is contained in the union of at most $\poly_{\format}(\degree,g,\log H)$ algebraic hypersurfaces of degree at most $\order[F]{1}\cdot g^{\dim X +1}(\log H)^{\dim X}$.
\end{theorem}
We prove \Cref{thm: interpolation} in \Cref{sec: algebraic point interpolation} below.

\begin{proof}[Proof of \Cref{thm: wilkie}]
	The proof follows the general strategy of the proof of the Pila--Wilkie theorem \cite{PilaWilkie2006} which is by now standard, and so we only recall the main points here.

	Fix $g,H\in\nNatural$.
	Since the height and degree of an algebraic point are preserved under coordinate inversion and negation, we may assume $X\subset[0,1]^n$.
	By \Cref{thm: sharp cd}, we may decompose $X$ into $\polyfd[\format][\degree]$ admissible $\LE$-cells in $\structure(\exp)_{\order[\format]{1},\polyfd[\format][\degree]}$.
	It therefore suffices to treat the case where $X$ itself is an admissible $\LE$-cell (by changing the polynomial in the right-hand side of \eqref{eq: Wilkie conjecture bound}).
	In particular, we have in this case that $X$ is pure dimensional.
	Denote its dimension by $\dim X$.

	If $\dim X=0$, then $X$ is a union of $\polyfd$ points and we are done, and so we assume otherwise from now.
	As $\format\geq \dim X$, it suffices to prove \Cref{thm: wilkie} with the right-hand side of \eqref{eq: Wilkie conjecture bound} replaced by $\poly_{\format,\dim X}(\degree,g,\log H)$.

	Using \Cref{thm: interpolation}, we find, for every $(\dim X+1)$-tuple of coordinate functions, a non-zero polynomial in these coordinates of degree at most $\poly_{\format,\dim X}(\degree,g,\log H)$ whose vanishing set contains $X(g,H)$.
	Intersecting the zero-sets of all of these polynomials yields an algebraic set $Y$ of dimension at most $\dim X$. 

	We consider the intersection $X\cap Y$, which is of format $\order[\format,\dim X]{1}$ and degree $\polyfd[\format,\dim X][\degree,g,\log H]$.
	By \Cref{thm: sharp cd}, we may decompose ${X\cap Y}$ into $\polyfd[\format,\dim X][\degree,g,\log H]$ admissible $\LE$-cells of format $\order[\format,\dim X]{1}$ and degree $\polyfd[\format,\dim X][\degree,g,\log H]$.
	Any of these cells which attains the maximal dimension $\dim X$ is contained in $X^{\operatorname{alg}}$ and may thus be discarded.
	For each of the remaining cells, we may proceed by induction on dimension.
\end{proof}

\subsection{Interpolating algebraic points}
\label{sec: algebraic point interpolation}
In this section, we prove \Cref{thm: interpolation}, as well as its $\Ranexp$ analogue, \Cref{thm:C}.
We follow an approach proposed by Wilkie in~\cite{Wilkie2015}.
See also~\cite[Section 6.1]{Binyamini2022} and~\cite[Appendix B]{BinyaminiNovikov2019} for arguments similar to those in this section.

We will need the following variant of the Thue--Siegel lemma.
For a matrix $A\in\operatorname{Mat}_{m\times n}(\rReal)$, write $\norm{A}_{\infty}$ for $\max_{1\leq i\leq m}\sum_{j=1}^n{\abs{A_{i,j}}}$.
\begin{lemma}[{\cite[Lemma 4.11]{Waldschmidt2000}}]
	\label{lem: siegel}
	Let $A\in\operatorname{Mat}_{m\times n}(\rReal)$ and $N\in\nNatural$.
	Then there exists a vector $0\neq v\in\zIntegers^n$, all of whose entries are bounded by $N+1$, such that the entries of $Av$ are all bounded by
	\begin{equation}
		N\cdot N^{-\frac{n}{m}}\cdot \norm{A}_{\infty}.
	\end{equation}
\end{lemma}

We now set up some notation to be used for the rest of this section.
Let $X\subset[-1,1]^n$ be an admissible $\LE$-cell in $\structure(\exp)_{\format,\degree}$ with $\dim X < n$.
Fix $H,g\in\nNatural$ and let $\boldsymbol{\delta}$ be as in \Cref{lem: thin annuli} with respect to parameters $t>\order*{1}$ and $0<\delta<\order{1}$.
The values of $t$ and $\delta^{-1}$ will be determined later, and will be taken to be bounded by $\polyfd[\format][\degree,g,\log H]$.
Applying \Cref{thm: single cell} to $X$ and the coordinate functions $x_1,\dots,x_n$, we reduce to the case where $X$ admits a recursive log--exp scale $S:X\to\rReal^\ell$, for $\ell=\order[\format]{1}$, such that the coordinate functions on $X$ satisfy $x_i=f_i\comp\phi^{-1}\comp S$, where $\phi:\ext{\cell}{\boldsymbol{\delta}}\to\cComplex^\ell$ is a power map on a real complex cell~$\ext{\cell}{\boldsymbol{\delta}}$ with $m\leq\dim X$ disc coordinates, $S(X)\subset\phi(\rRealPos\cell)$, and $f_i:\ext{\cell}{\boldsymbol{\delta}}\to\cComplex$ are holomorphic.
We assume that the components of $S$ are all in $\structure(\exp)_{\format,\degree}$ and that $\ext{\cell}{\boldsymbol{\delta}},\phi\in\FD[\format][\degree]$.
Finally, for $r\in\nNatural$ and $\alpha\in\zIntegers^r$, we write $\abs{\alpha}=\sum_{i=1}^r \abs{\alpha_i}$ and let $\operatorname{pos}(\alpha)$ be the vector whose $i$-th coordinate is~$\abs{\alpha_i}$.

\begin{remark}
	\label{rem: small functions in extension}
	In the construction of the complex cells $\ext{\cell}{\boldsymbol{\delta}}$ and the functions $f_i$ in the proof of the $\structure$-preparation theorem in~\cite[Theorem 5.3]{CarmonAnalyticallyGenerated}, one may take the functions $f_i$ to be either identically equal to $1$ or $0$ on $\ext{\cell}{\boldsymbol{\delta}}$ or else nowhere equal to $1$  or $0$ on $\ext{\cell}{\boldsymbol{\delta}}$.
  	Assume we are in this latter case and take $t,\delta^{-1}=\polyfd[\format][\degree]$.
  	We note that $\abs{f_i}\leq 1$ on $\phi^{-1}(S(X))$ and so, by~\cite[Lemma 55]{BinyaminiNovikov2019}, we have $\abs{f_i}\leq 2$ on $\cell$.
	In the same way, by first taking the cells $\cell$ to admit the larger extension $\ext*{\ext{\cell}{\boldsymbol{\delta}}}{\boldsymbol{{\delta}}}$, we may assume that $\abs{f_i}\leq 2$ on $\ext{\cell}{\boldsymbol{\delta}}$.
\end{remark}

Keeping the same notation as above, we have the following lemma.

\begin{lemma}
	\label{lem: upper bound}
	Fix $N,d\in\nNatural$ and let $f$ be an $(m+1)$-tuple of functions among $f_1,\dots,f_n$.
	For simplicity, assume these are $f_1,\dots,f_{m+1}$.
	One may choose parameters $t,\delta^{-1}=\polyfd[\ell][d,\log N]$ such that there exists a non-zero polynomial $P\in\zIntegers[\xi_1,\dots,\xi_{m+1}]$ of degree at most $d$ and with coefficients of size at most $N+1$, for which the function $\log(\abs{P\comp f})$, where defined, is bounded above by
	\begin{equation}
		\label{eq: log upper bound}
		\order[\ell]{\log N+d}-\order*[\ell]{1}\cdot d(\log N)^{\frac{1}{m+1}}
	\end{equation}
	on $\cell$.
\end{lemma}
\begin{proof}
Let $\alpha\in\nNatural^{m+1}$ with $\abs{\alpha}\leq d$.
We consider the Laurent expansion of $f^\alpha=\prod_i f_i^{\alpha_i}$ on~$\cell\subset\cComplex^\ell$ in terms of \emph{normalized monomials} (see~\cite[Section 3.2]{BinyaminiNovikov2019}):
\begin{equation}
	f^\alpha=\sum_{\beta\in \zIntegers^{\ell}} c_{\alpha,\beta}\cdot \varZ^{[\beta]}\qc c_{\alpha,\beta}\in\rReal.
\end{equation}
We have, by \cite[Proposition 29]{BinyaminiNovikov2019}, that $\norm{\varZ^{[\beta]}}_{\cell}\leq1$ and $\abs{c_{\alpha,\beta}}<\boldsymbol{\delta}^{\operatorname{pos}(\beta)}\cdot \norm{f^\alpha}_{\ext{\cell}{\boldsymbol{\delta}}}$.
Since each $f_i$ may be taken to be of size at most $2$ on $\ext{\cell}{\boldsymbol{\delta}}$ (see \Cref{rem: small functions in extension}), we may assume $\abs{c_{\alpha,\beta}}<\boldsymbol{\delta}^{\operatorname{pos}(\beta)}\cdot 2^d$.

Write $P=\sum_{\abs{\alpha}\leq d} \nu_\alpha \boldsymbol{\xi}^\alpha$, where $\boldsymbol{\xi}=(\xi_1,\dots,\xi_{m+1})$ and each $\nu_{\alpha}\in \zIntegers$ is of size at most~$N+1$.
We consider the expansion
\begin{equation}
	\label{eq: polynomial expansion all terms}
	P\comp f=\sum_{\abs{\alpha}\leq d}\nu_\alpha f^\alpha = \sum_{\substack{\abs{\alpha}\leq d,\\\beta\in\zIntegers^\ell}}\nu_\alpha \cdot c_{\alpha,\beta}\cdot\varZ^{[\beta]}.
\end{equation}
We wish to choose the coefficients $\nu_\alpha$, not all zero, such that \eqref{eq: polynomial expansion all terms} is uniformly small on $\cell$.
We split the right-hand side of \eqref{eq: polynomial expansion all terms} into three terms and bound each of them separately.
The three cases we consider are the following.
\begin{enumerate}
	\item Summands where $\abs{\beta}>t$;
	\item Summands where $\abs{\beta}\leq t$, and at least one of the non-disc coordinates of $\varZ$ appears with a non-zero exponent; and
	\item Summands where $\abs{\beta}\leq t$, and only disc coordinates of $\varZ$ appear with a non-zero exponent.
\end{enumerate}
All the bounds below are over $\cell$, on which we have $\vert{\varZ^{[\beta]}}\vert\leq 1$ for all $\beta$. 

A direct computation shows that
\begin{equation}
	\begin{split}
		\sum_{\abs{\beta}>t}  \delta^{\abs{\beta}} &< \sum_{\substack{T\in\nNatural\\T>t}} (2T+1)^\ell \delta^T\\
		&=O_\ell(t^\ell)\delta^t.
	\end{split}
\end{equation}
Therefore, for summands in \eqref{eq: polynomial expansion all terms} with $\abs{\beta}>t$, we have
\begin{align}
	\label{eq: high taylor coefficients bound}
	\begin{split}
		\Big\vert{\sum_{\substack{\abs{\alpha}\leq d,\\\abs{\beta}>t}}\nu_\alpha \cdot c_{\alpha,\beta}\cdot\varZ^{[\beta]}}\Big\vert
	&\leq 
	\sum_{\substack{\abs{\alpha}\leq d,\\\abs{\beta}>t}} N\cdot \delta^{\abs{\beta}}\cdot 2^d\\
	&\leq  (d+1)^{m+1} \cdot (N+1)\cdot 2^d \cdot O_\ell(t^\ell)\delta^t\\
	&=\polyl(N,t)\cdot\order[\ell]{1}^d\cdot\delta^t.
	\end{split}
\end{align}

For summands in \eqref{eq: polynomial expansion all terms} with $\abs{\beta}\leq t$, such that at least one of the non-disc coordinates of $\varZ$ appears with a non-zero exponent, we have that $\boldsymbol{\delta}^{\operatorname{pos}(\beta)}$ is, by definition of $\boldsymbol{\delta}$, at most $\delta^{t}$.
We may thus also bound this part of \eqref{eq: polynomial expansion all terms} by
\begin{equation}
	\label{eq: big extension bound}
	(d+1)^{m+1}\cdot (2t+1)^\ell\cdot (N+1)\cdot 2^d\cdot\delta^t=\polyl(N,t)\cdot\order[\ell]{1}^d\cdot\delta^t.
\end{equation}
It is in this part where we crucially use the theory of complex cells and the fact that we may take the extension parameter of non-disc coordinates to be exponentially smaller than in the disc coordinates.

It thus remains to bound the part of \eqref{eq: polynomial expansion all terms} corresponding to $\abs{\beta}\leq t$, where only variables corresponding to one of the $m$ disc fibers of $\cell$ may appear with a non-zero exponent.
We assume for convenience that these are the coordinates $z_1,\dots,z_m$ and that $\beta\in\nNatural^m$.
We have
\begin{align}
	\Big\vert{\sum_{\substack{\abs{\alpha}\leq d,\\\abs{\beta}\leq t}}\nu_\alpha \cdot c_{\alpha,\beta}\cdot\varZ^{[\beta]}}\Big\vert 
	&\leq   (t+1)^{m} \max_{\abs{\beta}\leq t}\Big\vert{\sum_{\abs{\alpha}\leq d}\nu_\alpha \cdot c_{\alpha,\beta}}\Big\vert 
\end{align}
and so we wish to bound each of the sums $\Big\vert{\sum_{\abs{\alpha}\leq d}\nu_\alpha \cdot c_{\alpha,\beta}}\Big\vert$ for all $\abs{\beta}\leq t$.
There are $\binom{d+m+1}{m+1}\sim d^{m+1}$ choices for vectors $\alpha$ with $\abs{\alpha}\leq d$, up to multiplicative constants depending on $\ell$. 
Similarly, there are $\binom{t+m}{m}\sim t^m$ sums that we are trying to simultaneously bound, up to multiplicative constants depending on $\ell$.

Applying \Cref{lem: siegel} to the transpose of the matrix $(c_{\alpha,\beta})_{\alpha,\beta}$, we find coefficients $\nu_\alpha$ of size at most $N+1$, not all of them zero, such that
\begin{equation}
	\label{eq: siegel application}
	\Big\vert{\sum_{\abs{\alpha}\leq d}\nu_\alpha \cdot c_{\alpha,\beta}}\Big\vert \leq N \cdot N^{-\order*[\ell]{1}\cdot\frac{d^{m+1}}{t^m}} \cdot 2^d \cdot (d+1)^{m+1}
\end{equation}
for all $\abs{\beta}\leq t$.

Choosing $\delta=\order[\ell]{1}$ and $t=d(\log N)^{\frac{1}{m+1}}$, the estimates \eqref{eq: high taylor coefficients bound} to \eqref{eq: siegel application} give
\begin{equation}
	\label{eq: bounds before logs}
	\abs{P \comp f}
	\leq
	\polyl(N)\cdot\order[\ell]{1}^d\cdot e^{-\order*[\ell]{1} \cdot d(\log N)^{\frac{1}{m+1}}}.
\end{equation}
Taking logarithms in \eqref{eq: bounds before logs} now yields \eqref{eq: log upper bound}.
\end{proof}

In the notation of \Cref{lem: upper bound} above, we compare the upper bound \eqref{eq: log upper bound} with the following Liouville lower bound for $\abs{P\comp f}$.
\begin{lemma}[{\cite[Lemma 14]{Habegger2018}}]
	\label{lem: Liouville bound}
	Assume $\xi_1,\dots,\xi_{m+1}\in\overline{\qRationals}$ are of degree at most~$g$ and multiplicative Weil height at most $H$.
	Then either $P(\boldsymbol{\xi})=0$, or else
	\begin{equation}
		\label{eq: liouville inequality}
		\abs{P(\boldsymbol{\xi})} \geq (d^{m+1} N H ^{d(m+1)})^{-g}.
	\end{equation}
\end{lemma}

When $P(\boldsymbol{\xi})\ne 0$ in \Cref{lem: Liouville bound}, taking logarithms in~\eqref{eq: liouville inequality} gives
\begin{equation}
	\label{eq: log lower bound}
	\begin{split}
		\log(\abs{P(\boldsymbol{\xi})}) &\geq -g((m+1)\log d+\log N+d(m+1)\log H)\\
		&\geq -\order*[\ell]{1}\cdot g\cdot(\log N+ d\log H).
	\end{split}
\end{equation}
Comparing \eqref{eq: log upper bound} and \eqref{eq: log lower bound}, a direct computation shows that ${P\comp f}$ must vanish on those points of $\cell$ where the components $f_1,\dots,f_{m+1}$ of $f$ are algebraic of degree at most $g$ and height at most $H$ if we set 
\begin{equation}
	N = H^d \qc d\geq \order*[\ell]{g^{m+1} (\log H)^m}.
\end{equation}
In particular, this may be done with $d \leq \order[\ell]{g^{m+1} (\log H)^m}$.
Composing $P$ with the coordinate functions on the $\LE$-cell $X$ which correspond to the components $f_1,\dots,f_{m+1}$ of $f$ finishes the proof of \Cref{thm: interpolation}. 
\qed

\subsubsection{The case of \texorpdfstring{$\Ranexp$}{Ran,exp}}

The proof of \Cref{thm:C} is similar to that of \Cref{thm: interpolation}.
We now explain the necessary changes.
First, we replace throughout the proof all bounds of the form $\poly_{\format}(\degree,\cdot)$ or $\order[\format]{1}$ by $\poly_{X}(\cdot)$ or $\order[X]{1}$, respectively.
In addition, \Cref{lem: thin annuli} is replaced by its $\Ran$ analogue~\cite[Lemma 94]{BinyaminiNovikov2019}.

The main point is the following.
Instead of \Cref{thm: single cell}, we use \Cref{thm: single cell Ranexp} to reduce to the case where the coordinate functions $x_1,\dots,x_n$ on each member $X_t$ of the family $X\subset T\cross [0,1]^n$ (or, more precisely, their restrictions to each element in a decomposition of $X_t$ into $\LE$-cells) are prepared via complex cells with respect to a common recursive $\log$--$\exp$-scale.
We now proceed as in the proof of \Cref{thm: interpolation}, for each set $X_t$ separately --- this ensures that we construct interpolation hypersurfaces for each $X_t$ rather than just for those $t$ which are themselves algebraic of bounded height and degree.

The uniform bounds (over $T$) on the number of resulting $\LE$-cells and the lengths of the relevant recursive $\log$--$\exp$-scales, as well as the uniform bounds on the sufficient sizes of $t,\delta^{-1}$ in \Cref{rem: small functions in extension}, yield uniform bounds for the number of resulting algebraic hypersurfaces and their degrees, which concludes the proof of \Cref{thm:C}. \qed

\bibliographystyle{abbrv}
\bibliography{references}
\end{document}